\documentclass[12pt]{amsart}
\usepackage{fullpage}
\usepackage[utf8]{inputenc}
\usepackage[english]{babel}
\usepackage{mathtools}
\usepackage{amssymb}
\usepackage{amsthm}
\usepackage{mathrsfs}
\usepackage{nicefrac}
\usepackage{upgreek}
\usepackage[T1]{fontenc}
\usepackage{graphicx}
\usepackage{hyperref}
\usepackage{stmaryrd}
\usepackage{graphicx}
\usepackage{booktabs}
\usepackage{hyperref}

\usepackage[all]{xy}
\usepackage{graphics}
\usepackage{color}
\usepackage[T1]{fontenc}
\usepackage{parskip}

\usepackage{enumitem}

\makeatletter
\newtheorem*{rep@theorem}{\rep@title}
\newcommand{\newreptheorem}[2]{%
\newenvironment{rep#1}[1]{%
 \def\rep@title{#2 \ref{##1}}%
 \begin{rep@theorem}}%
 {\end{rep@theorem}}}
\makeatother

\newcommand{\dl}{\langle\langle}
\newcommand{\dr}{\rangle\rangle}

\newcommand{\C}[1]{{\mathcal #1}}

\newcommand{\bracenom}{\genfrac{\llbracket}{\rrbracket}{0pt}{}}

\newtheorem{theorem}{Theorem}[section]
\newtheorem{lemma}[theorem]{Lemma}
\newtheorem{corollary}[theorem]{Corollary}
\newtheorem{proposition}[theorem]{Proposition}

\theoremstyle{definition}\newtheorem{definition}[theorem]{Definition}
\theoremstyle{remark}\newtheorem{remark}[theorem]{Remark}
\theoremstyle{definition}
\makeatletter\@addtoreset{case}{example}\makeatother
\theoremstyle{definition}

\begin{document}

\title{Bounded generation by root elements for Chevalley groups defined over rings of integers of function fields with an application in strong boundedness}

\author{Alexander A. Trost}
\address{Fakult\"{a}t f\"{u}r Mathematik, Ruhr Universit\"{a}t Bochum, D-44780 Bochum, Germany}
\email{Alexander.Trost@ruhr-uni-bochum.de}

\begin{abstract}
Bounded generation by root elements is a property which has been widely studied for various types of linear algebraic groups defined over rings of integers in algebraic number fields \cite{MR2357719,MR3892969,MR704220,MR1044049}. However, when considering global function fields, there are not many results beyond the treatment of two special cases due to Nica \cite{Nica} and Queen \cite{MR379441}. In this paper, we use model theoretic methods due to Carter, Keller and Paige written up by Morris \cite{MR2357719} to prove bounded generation by root elements for simply connected, split Chevalley groups defined over the ring of all integers in a global function field. We further apply this bounded generation result together with previous results from our paper \cite{General_strong_bound} to derive that the aforementioned Chevalley groups satisfy the strong boundedness property introduced by Kedra, Libman and Martin in \cite{KLM}.
\end{abstract}

\maketitle

\section{Introduction} 
\label{intro}

A group $G$ is called \textit{boundedly generated} if there are finitely many cyclic groups $Z_1,\dots,Z_L$ such that $G=Z_1\cdots Z_L.$ This property was introduced, among else, as a possible unified approach to proving the Congruence Subgroup Property of Arithmetic Groups as stated in the (now disproven) Rapinchuk-Conjecture \cite[Question~A]{corvaja2021nonvirtually} and has applications concerning the representation theory of the group in question. There are many papers proving this property for various classical, arithmetic groups like ${\rm SL}_n(\mathbb{Z})$ in a paper by Carter and Keller \cite{MR704220} or for more general arithmetic (twisted) Chevalley groups by Tavgen \cite{MR1044049}. For ${\rm SL}_2(R)$ for $R$ a ring of algebraic integers with infinitely many units there are papers by Carter, Keller and Paige written up by Morris \cite{MR2357719} and another one by Rapinchuk, Morgan and Sury \cite{MR3892969}. More recently, there was also a result for certain isotropic orthogonal groups over number fields \cite{MR2228948}. In case of classical arithmetic, matrix groups like ${\rm SL}_n$ or ${\rm Sp}_{2n}$, bounded generation results are usually proven by showing that each element of the group can be written as some bounded product of elementary matrices (or the corresponding generalization in the group, so called \textit{root elements}). This commonly referred to as \textit{bounded generation by root elements} or \textit{bounded elementary generation}.

However, essentially all of the aforementioned results concerning bounded generation by root elements are concerned with groups defined using algebraic number fields and there seem to be few comparable results for the case of global function fields besides the following two: First, there is a paper due to Nica stating the following:

\begin{theorem}\cite[Theorem~2]{Nica}\label{Nica_thm}
Let $\mathbb{F}$ be a finite field and $n\geq 3.$ Then ${\rm SL}_n(\mathbb{F}[T])$ is boundedly generated by root elements.
\end{theorem}

Nica's result is essentially an adaptation of the result by Carter and Keller for bounded elementary generation in the number field case \cite{MR704220} and uses that $\mathbb{F}_q[T]$ is a principal ideal domain to avoid the more involved techniques employed by Carter and Keller.
Second, there is a partial result concerning certain localizations of the ring of all integers in global function fields, namely:

\begin{theorem}\cite[Theorem~2]{MR379441}
Let $K$ be a global function field that is a finite extension of $\mathbb{F}_q(T)$ with ring of integers $\C O_K$ and field of constants $\mathbb{F}_q.$ Further, let $S$ be a finite set of maximal ideals in $\C O_K$ such that $\gcd\{{\rm deg}(\mathfrak{P})\mid\mathfrak{P}\in S\}=1$ and such that $|S|\geq 2$. Also assume that the class number of $\C O_K$ as a Dedekind domain is coprime to $q-1$ or that $S$ contains an element of degree $1.$ Set 
\begin{equation*}
R:=\{a/b\mid a\in\C O_K,b\neq 0\text{ and }\{\text{prime divisors of }bR\}\subset S\}.
\end{equation*}
Then ${\rm SL}_2(R)$ is boundedly generated by root elements.
\end{theorem}

One of the main goals of this preprint is to generalize such results to all rings of integers in global function fields: 

\begin{theorem}\label{main_thm}
Let $K$ be a global function field that is a finite extension of $\mathbb{F}_q(T)$ with ring of integers $R:=\C O_K$ and field of constants $\mathbb{F}_q.$ Further, let $\Phi$ be an irreducible root system of rank at least $2.$ Then $G(\Phi,R)$ is boundedly generated by root elements. More precisely, for $k:=[K:\mathbb{F}_q(T)],$ there is a constant $L(q,k)$ only depending on $q$ and $k$ and not on $K$ such that each element of $G(\Phi,R)$ can be written as a product of $L(q,k)|\Phi|$ many root elements. 
\end{theorem}

The techniques from \cite{MR704220} are cumbersome to adapt to the general global function field setup. Instead Theorem~\ref{main_thm} will be shown by using Morris' approach \cite{MR2357719} for the global function field case. Morris' arguments are model theoretic in nature. Said arguments proceed by deriving bounded generation by root elements for ${\rm SL}_n(R)$ for $R$ the ring of all algebraic integers of a number field from a first order compactness argument applied to a first order theory of commutative rings satisfied by all rings of all algebraic integers of a number field.

Our proof of Theorem~\ref{main_thm} at least for $\Phi=A_n$ (that is ${\rm SL}_{n+1}$) essentially consists of deriving a modified version of said theory satisfied by the rings of all integers in a global function field $K.$ However, new arguments have to be found to derive bounded generation results for some of the other root systems: Assuming a bounded generation result for the cases of $\Phi=C_2$ and $G_2$, the general Theorem~\ref{main_thm} for all root systems $\Phi$ can be proven by induction. Thus one is left with dealing with proving bounded generation for ${\rm Sp}_4(R)$ and $G_2(R).$ The latter case can be easily reduced to the ${\rm SL}_3(R)$-case, but the ${\rm Sp}_4$-case requires a different argument, which consists of a model theoretic argument inspired in part by the proof of Bass, Milnor and Serre's \cite[Theorem~3.7]{MR244257}. 

Next, the reason why we are interested in bounded generation by root elements for $G(\Phi,R)$ in the first place is the fact that it can be used to derive strong boundedness for the corresponding $G(\Phi,R).$ A group $G$ is called \textit{strongly bounded} if for each finite collection of conjugacy classes $T$ generating $G,$ there is a natural number $L(G,|T|)$ such that each element of $G$ can be written as a product of $L(G,|T|)$ many elements of $T\cup T^{-1}.$ The smallest such $L(G,|T|)$ is denoted by $\Delta_{|T|}(G).$ Strong boundedness was initially introduced by Kedra, Libman and Martin in \cite{KLM} and we showed in our previous paper \cite{General_strong_bound} that this property holds for essentially all arithmetic Chevalley groups:  

\begin{theorem}\cite[Theorem~5.13]{General_strong_bound}\label{alg_integers_strong_bound}
Let $R$ be the ring of all S-algebraic integers in a number field and $\Phi$ an irreducible root system of rank at least $2$. Then there is a constant $C(\Phi,R)\geq 1$ such that $\Delta_l(G(\Phi,R))\leq C(\Phi,R)\cdot l$ holds for all $l\in\mathbb{N}.$
\end{theorem}

The proof of this theorem is done by an application of a compactness argument to older results concerning the structure of normal subgroups of the $G(\Phi,R)$ followed by an application of bounded generation by root elements. Having obtained a bounded generation result in the form of Theorem~\ref{main_thm}, one can derive a corresponding theorem for global function fields:

\begin{theorem}\label{strong_bound_pos}
Let $K$ be a global function field that is a finite extension of $\mathbb{F}_q(T)$ with ring of integers $R:=\C O_K$ and field of constants $\mathbb{F}_q.$ Further, let $\Phi$ be an irreducible root system of rank at least $2$ and set $k:=[K:\mathbb{F}_q(T)].$ Then there is a constant $C(\Phi,R)\in\mathbb{N}$ such that 
$\Delta_l(G(\Phi,R))\leq C(\Phi,q,k)\cdot l$ holds for all $l\in\mathbb{N}.$ Additionally, if $\Phi\neq G_2$ or $C_2,$ the constant $C(\Phi,R)$ only depends on $q,k$ and $\Phi.$ In particular, $G(\Phi,R)$ is strongly bounded for all irreducible $\Phi$ of rank at least $2$.
\end{theorem}

The main problem in proving Theorem~\ref{strong_bound_pos} comes from the specific cases of $\Phi=C_2$ or $G_2$ as the proof of Theorem~\ref{alg_integers_strong_bound} uses the fact that $2R$ has finite index in $R$ for this cases, which can clearly fail if the global function field in question has characteristic $2.$ Last, we present two corollaries which describe collections of conjugacy classes in $G(\Phi,R)$ with diameters proportional to the size of the corresponding collections. These two corollaries are:

\begin{corollary}\label{lower_bounds}
Let $K$ be a global function field that is a finite extension of $\mathbb{F}_q(T)$ with ring of integers $R:=\C O_K.$ Further, let $\Phi$ be an irreducible root system of rank at least $2$, which is not $\Phi=C_2$ or $G_2$. Then $\Delta_l(G(\Phi,R))\geq l$ holds for all $l\in\mathbb{N}.$
\end{corollary}

and

\begin{corollary}\label{lower_bounds_rank2}
Let $\Phi$ be $C_2$ or $G_2$ and let $K$ be a global function field with ring of integers $R:=\C O_K$. Further, let
\begin{equation*}
r:=r(R):=|\{\C P\mid \C P\text{ is a prime ideal and }R/\C P=\mathbb{F}_2\}|
\end{equation*}
be given. Then for $G(\Phi,R)$
\begin{enumerate}
\item{the inequality $\Delta_l(G(\Phi,R))\geq l$ holds for all $l\geq r(R)$ and}
\item{the equality $\Delta_l(G(\Phi,R))=-\infty$ holds for $l<r(R).$} 
\end{enumerate}
\end{corollary}

The rest of the paper is divided into five additional sections: In the second section, we introduce all concepts used in the rest of the paper and state some lemmas needed later on. In particular, we will recall the definition of global function fields and their rings of integers, define Mennicken groups and give a short account of simply-connected, split Chevalley groups and conjugation generated word norms. In the third section, we derive Theorem~\ref{main_thm} from a theorem of Morris and in the fourth section we derive Theorem~\ref{strong_bound_pos}. The fifth section in turn will prove Corollary~\ref{lower_bounds} and \ref{lower_bounds_rank2}. The last section contains a couple of remarks concerning generalizations of Theorem~\ref{main_thm} and a (short) discussion of another possible application of said theorem.

\section*{Acknowledgments}

This work was funded by the LMS Early Career Fellowship.

\section{Basic definitions and notions}
\label{sec_basic_notions}

\subsection{Global function fields and their integers}

In this subsection, we introduce global function fields, their rings of integers and their primes. For a good introduction to the topic, we refer the reader to Rosen's book \cite{MR1876657}. We begin with the definition of global function fields:

\begin{definition}
A \textit{global function field over $\mathbb{F}_q$} is a finite field extension $K$ of the rational function field $\mathbb{F}_q(T)$ for $\mathbb{F}_q$ a finite field with $q$ elements. Then the ring of integers $\C O_K$ of $K$ is defined as the integral closure of $\mathbb{F}_q[T]$ in $K.$
\end{definition}

\begin{remark}
\hfill
\begin{enumerate}
\item{Equivalently, one can define global function fields as the fraction field of functions $\mathbb{F}_q(C)$ of a nonsingular, geometrically integral, affine curve $C$ defined over the finite field $\mathbb{F}_q.$}
\item{The principal example of a global function over $\mathbb{F}_q$ is the field of rational functions $\mathbb{F}_q(T).$ The associated curve is the affine line $\mathbb{A}^1$ of the field $\mathbb{F}_q.$}
\end{enumerate}
\end{remark}

In the rest of the paper, we will always replace the field $\mathbb{F}_q$ by its algebraic closure $\mathbb{F}$ in $K$. We leave it as an exercise to check that $\mathbb{F}$ is a finite field extension of $\mathbb{F}_q$. Assuming that $\mathbb{F}_q$ is algebraically closed in $K$, we call $\mathbb{F}_q$ the \textit{field of constants of $K$}. Later, we will also need the concept of primes in global function fields, not to be confused with the prime ideals of the corresponding ring of integers:
 
\begin{definition}
Let $K$ be a global function field with field of constants $\mathbb{F}_q$. 
\begin{enumerate}
\item{Then a \textit{prime of $K$} is a subring $R_{\C P}$ of $K$ containing $\mathbb{F}_q$, such that $R_{\C P}$ is a discrete valuation domain with fraction field $K$ and maximal ideal $\C P.$}
\item{We define the function
\begin{equation*}
{\rm ord}_{\C P}:K^*\to\mathbb{Z}
\end{equation*}  
as follows: If $a\in R_{\C P}-\{0\}$ is given, then ${\rm ord}_{\C P}(a)\in\mathbb{Z}$ is defined as the maximal $n\in\mathbb{N}_0$ such that $a\in\C P^n.$ If $a\in K^*$ is not an element of $R,$ then $a^{-1}$ is an element of $R_{\C P}$ and then ${\rm ord}_{\C P}(a)$ is defined as $-{\rm ord}_{\C P}(a^{-1}).$}
\item{For a given prime $R_{\C P}$ of $K,$ the \textit{degree} of said prime, denoted by ${\rm deg}(\C P),$ is defined as ${\rm deg}(\C P):=[R_{\C P}/\C P:\mathbb{F}_q].$ }
\end{enumerate}
\end{definition}

\begin{remark}
Especially later on, we will often refer to `the prime $\C P$ of $K$' rather than `to the prime $R_{\C P}$ of $K$.'
\end{remark}

Furthermore, for a finite extension $K\mid L$ of global function fields and $R_{\C P}$ a prime of $K,$ the intersection $S_p:=R_{\C P}\cap L$ is a prime of the global function field $L.$ In this case, we refer to the prime $\C P$ of $K$ as covering (or lying over) the prime $p$ of $L.$ In this case, let $x_p\in S_p$ be the generator of the maximal ideal $p$ in $S_p.$ We then define $e(\C P,p):={\rm ord}_{\C P}(x_p)$. The integer $e(\C P,p)$ does not depend on the choice of $x_p$ and is called the \textit{ramification index of $\C P$ over $p.$} Furthermore, the integer $f(\C P\mid p):=[R_{\C P}/\C P:S_p/p]$ is called \textit{the relative degree of $\C P$ over $p.$} The following theorem holds:

\begin{theorem}\cite[Theorem~7.6]{MR1876657}\label{sum_formula}
Let $K\mid L$ be a finite extension of global function fields and $p$ a prime of $L.$ Then $T_p$, the set of all primes of $K$ covering $p$, is finite and the equation
$[L:K]=\sum_{\C P\in T_p}f(\C P\mid p)\cdot e(\C P\mid p)$ holds.
\end{theorem}

Also the following statement holds:

\begin{lemma}\cite[Proposition~5.1]{MR1876657}\label{class_field_formula}
Let $K$ be a global function field. Then for all $a\in K^*$, there are only finitely many primes $\C P$ of $K$ such that ${\rm ord}_{\C P}(a)\neq 0.$ Further, the equation $\sum_{\C P}{\rm deg}(\C P)\cdot{\rm ord}_{\C P}(a)=0$ holds with the sum taken over all primes of $K.$ 
\end{lemma}

As mentioned already, the field $\mathbb{F}_q(T)$ is a global function field in its own right. It has two distinct types of primes: First, for a monic, irreducible polynomial $f(T)\in\mathbb{F}_q[T]$, consider the localization $R_f:=\mathbb{F}_q[T]_{f(T)\mathbb{F}_q[T]}.$ This local ring is a prime of $\mathbb{F}_q(T).$ Second, consider the ring $A':=\mathbb{F}_q[T^{-1}]$ and its localization $R_{\infty}:=A'_{T^{-1}A'}.$ The ring $R_{\infty}$ is also a prime of $\mathbb{F}_q(T)$ and is commonly called \textit{the prime $\C P_{\infty}$ at infinity of $\mathbb{F}_q(T)$}. We leave it as an exercise to the reader to check that these two types make up all primes of $\mathbb{F}_q(T)$ and that the degree of the prime $R_f$ for $f$ monic and irreducible is the degree of $f$ as a polynomial and that the degree of $\C P_{\infty}$ is equal to $1$. 

Also note, that each maximal ideal $\mathfrak{P}$ of the ring of integers $\C O_K$ of a global function field $K$ defines a prime of $K$, namely the localization $R_{\mathfrak{P}}:=(\C O_K)_{\mathfrak{P}}.$ Additionally, the isomorphism $R_{\mathfrak{P}}/(\mathfrak{P}R_{\mathfrak{P}})=\C O_K/\mathfrak{P}$ implies that ${\rm deg}(R_{\mathfrak{P}})={\rm dim}_{\mathbb{F}_q}(\C O_K/\mathfrak{P}).$ We denote this number by ${\rm deg}(\mathfrak{P}).$ We leave it as an exercise to check that the primes of $K$ arising in this manner are precisely the ones not covering the prime $R_{\infty}$ of $\mathbb{F}_q(T)$ and that the map 
\begin{equation*}
{\rm MaxSpec}(\C O_K)\to\{\C P\text{ prime of }K\mid \C P\text{ does not cover }R_{\infty}\},\mathfrak{P}\mapsto R_{\mathfrak{P}}
\end{equation*}
is a bijection. By way of this bijection, we will usually identify maximal ideals of $\C O_K$ with their associated primes. In this context, the following lemma is useful later in the paper: 

\begin{lemma}\label{degree_func}
Let $K$ be a global function field that is a finite extension of $\mathbb{F}_q(T)$ and has ring of integers $\C O_K$. Next, let $S_{\infty}$ be the set of primes of $K$ covering the prime $R_{\infty}$ of $\mathbb{F}_q(T)$. Then 
\begin{equation*}
{\rm deg}(a\C O_K)=-\sum_{\C P\in S_{\infty}} f(\C P\mid R_{\infty})\cdot{\rm ord}_{\C P}(a)
\end{equation*}
holds for each principal, maximal ideal $a\C O_K$ of $\C O_K.$ 
\end{lemma}

\begin{proof}
Set $\mathfrak{P}:=a\C O_K.$ Then the only primes $\C Q$ of $K$ for which ${\rm ord}_{\C Q}(a)\neq 0$ is possible are elements of $S_{\infty}\cup\{\mathfrak{P}\}.$ Next, note that by Theorem~\ref{class_field_formula}, the equation 
\begin{equation*}
0=\sum_{\C Q\text{ prime of }K}{\rm deg}(\C Q)\cdot{\rm ord}_{\C Q}(a)={\rm deg}(\mathfrak{P})\cdot{\rm ord}_{\mathfrak{P}}(a)
+\sum_{\C P\in S_{\infty}}{\rm deg}(\C P)\cdot{\rm ord}_{\C P}(a) 
\end{equation*} 
holds. However, obviously ${\rm deg}(\mathfrak{P})={\rm deg}(a\C O_K)$ and ${\rm ord}_{\mathfrak{P}}(a)=1$ holds. Thus we obtain 
\begin{equation*}
{\rm deg}(a\C O_K)=-\sum_{\C P\in S_{\infty}}{\rm deg}(\C P)\cdot{\rm ord}_{\C P}(a).
\end{equation*} 
Last, note that as $R_{\infty}$ has degree $1$ as a prime of $\mathbb{F}_q(T),$ we obtain 
\begin{equation*}
{\rm deg}(\C P)=[R_{\C P}/\C P:\mathbb{F}_q]=[R_{\C P}/\C P:R_{\infty}/\C P_{\infty}]\cdot[R_{\infty}/\C P_{\infty}:\mathbb{F}_q]=f(\C P\mid R_{\infty})\cdot{\rm deg}(\C P_{\infty})=f(\C P\mid R_{\infty})
\end{equation*}
for any prime $\C P\in S_{\infty}.$ This finishes the proof.
\end{proof}

\subsection{Mennicken symbols and the universal Mennicken group}

Mennicken symbols and the universal Mennicken group will play an important role in Section~\ref{proof_main}. These symbols were initially defined by Mennicken and subsequently used to great effect by Bass, Milnor and Serre \cite{MR244257} to investigate the Congruence Subgroup Property for arithmetic Chevalley groups. We first introduce the set $W(I):$

\begin{definition}
Let $R$ be a commutative ring with $1$ and $I$ a non-zero ideal of $R.$ 
\begin{enumerate}
\item{
Then the set $W(I)$ is defined as 
\begin{equation*}
W(I)=\{(b,a)\in R^2\mid (b,a)\equiv (0,1)\text{ mod }I, aR+bR=R\}.
\end{equation*}}
\item{For $(b,a)\in W(I)$, an element $(d,c)$ is called \textit{$I$-equivalent to $(b,a)$}, if there is an $x\in R$ or $y\in I$, such that either $c=a+xb$ or $d=b+ya.$}
\end{enumerate}
\end{definition}

\begin{remark}
Beware that the order of the condition I require of elements of $W(I)$ is flipped when compared to the classical definition given, say in the Bass, Milnor, Serre paper \cite{MR244257}; that is I require $(b,a)\equiv (0,1)\text{ mod }I$ rather than $(a,b)\equiv (1,0)\text{ mod }I$. This is mostly done to avoid having to switch the coordinates of an element $(b,a)\in W(I)$, when applying a Mennicken symbol to them. 
\end{remark}

Next, we define Mennicken symbols:

\begin{definition}
Let $R$ be a commutative ring with $1$ and $I$ a non-zero ideal of $R$ and $G$ a group. Then a map $f:W(I)\to G,(b,a)\mapsto f(b,a)$ is called a Mennicken symbol, if it satisfies the following properties 
\begin{enumerate}
\item{${\rm MS}(1):$ If $(b,a),(d,c)\in W(I)$ are $I$-equivalent, then $f(b,a)=f(d,c)$ holds.}
\item{${\rm MS}(2):$ If $(b_1,a),(b_2,a)$ are elements of $W(I),$ then 
\begin{equation*}
f(b_1\cdot b_2,a)=f(b_1,a)\cdot f(b_2,a)
\end{equation*}
holds.}
\end{enumerate}
\end{definition}

Using the familiar construction of universal objects in abstract algebra, there is a universal Mennicken group as follows:

\begin{proposition}\label{universal_mennicken}\cite[Lemma~2.19(5)]{MR2357719}
Let $R$ be a commutative ring with $1$ with stable range at most $3/2$ and $I$ a non-zero ideal in $R.$ Then there is an abelian group $M(I)$ and a Mennicken symbol 
$\bracenom{\:}{\:}_I:W(I)\to M(I),(b,a)\mapsto\bracenom{b}{a}_I$ satisfying the following universal property: For each group $G$ and each Mennicken symbol $f:W(I)\to G$, there is a unique homomorphism $\theta_f:M(I)\to G$ such that $\theta_f\circ\bracenom{\:}{\:}_I=f$. The group $M(I)$ is called \textit{the universal Mennicken group (of the pair $(R,I)$)}. 
\end{proposition}

\begin{remark}
Stable range will only be defined in Definition~\ref{first_order_def}.
\end{remark}

\subsection{Simply-connected split Chevalley groups and their root elements}\label{Chevalley}

We will not give a complete definition of simply-connected split Chevalley groups and their root elements. An account of them as $\mathbb{Z}$-defined linear group schemes is given in \cite[Section~2.1]{General_strong_bound}. For a more complete description please consider also \cite{MR1611814} and \cite[Theorem~1, Chapter~1, p.7; Theorem~6(e), Chapter~5, p.38; Lemma~27, Chapter~3, p.~29]{MR3616493}. 

For the purposes of this paper, we merely note that for each irreducible root system $\Phi$ there is a unique simply-connected, split, $\mathbb{Z}$-defined group scheme $G(\Phi,\cdot)$ such that for an algebraically closed field $K$ the root system of the Lie algebra associated to the linear algebraic group $G(\Phi,K)$ is (isomorphic to) $\Phi.$ There is a well-known classification of all irreducible root systems in terms of Dynkin-diagrams \cite[Appendix]{MR0396773}. The group schemes $G(\Phi,\cdot)$ encompass many classical matrix groups: For example, $G(A_n,R)={\rm SL}_{n+1}(R)$ and $G(C_n,R)={\rm Sp}_{2n}(R)$ holds for a commutative ring $R$ and $n\geq 2.$ 

For us the important part is that $G(\Phi,\cdot)$ can also be defined using polynomial equations defined over $\mathbb{Z}$: For each irreducible, root system $\Phi$, there is a positive integer $n_{\Phi}$ and a finite set of polynomials $P\subset\mathbb{Z}[y_{i,j}\mid 1\leq i,j\leq n_{\Phi}]$ such that 
\begin{equation*}
G(\Phi,R)=\{A\in R^{n_{\Phi}\times n_{\Phi}}\mid \forall p\in P:p(a_{i,j})=0\}
\end{equation*}
holds. The important point is that membership in $G(\Phi,R)$ is a property of commutative rings with $1$ describable in first order terms. The equations given by the set $P$ correspond to a faithful, linear representation of the linear algebraic group in question. For example, for $\Phi=C_2$ or $G_2$ the classical representations of $G(C_2,R)={\rm Sp}_4(R)\subset{\rm GL}_4(R)$ or $G(G_2,R)=G_2(R)\subset{\rm GL}_8(R)$ found in the papers \cite{MR1162432,MR1487611} by Keller and Costa can be used. These descriptions as matrix groups can also be used to define so-called level ideals:

\begin{definition}
\label{central_elements_def}
Let $R$ be a commutative ring with $1$ and let $A\in G(\Phi,R)$ be given for $\Phi=C_2$ or $G_2$. The \textit{level ideal $l(A)$} is defined  
\begin{enumerate}
\item{in case $\Phi=C_2$ as $l(A):=(a_{i,j},(a_{i,i}-a_{j,j})|1\leq i\neq j\leq 4).$}
\item{in case $\Phi=G_2$ as $l(A):=(a_{i,j},(a_{i,i}-a_{j,j})|1\leq i\neq j\leq 8).$}
\end{enumerate}
For $T\subset G(\Phi,R),$ we define $l(T):=\sum_{A\in T}l(A).$
\end{definition}

\begin{remark}
There are of course definitions of level ideals for the other root systems, too, but the two cases of $\Phi=C_2$ and $G_2$ are the only ones which are needed later in the paper.
\end{remark}

For a ring homomorphism $f:R\to S$, we denote the corresponding group homomrphism $G(\Phi,R)\to G(\Phi,S)$ also by $f.$ In this context, for a proper ideal $I$ of $R,$ we denote the group homomorphism $G(\Phi,R)\to G(\Phi,R/I)$ induced by the quotient map $R\to R/I$ by $\pi_I.$

Next, we note that fixing a ring $R$ there are specific elements called \textit{root elements} of $G(\Phi,R)$. We will not define them either, again referring to \cite[Section~2.1]{General_strong_bound} for more details, but note that for each root $\phi\in\Phi$ there is a group isomorphism
\begin{equation*}
\varepsilon_{\phi}:(R,+)\to G(\Phi,R),x\mapsto\varepsilon_{\phi}(x)
\end{equation*}
definable in first order terms. The elements $\varepsilon_{\phi}(x)$ are called root elements of $G(\Phi,R)$ and the subgroup $\{\varepsilon_{\phi}(x)\mid x\in R\}$ is called the \textit{root subgroup associated to the root $\phi.$} Supressing the root system $\Phi$ and the ring $R$ in question, we usually denote the set of all root elements of $G(\Phi,R)$ by ${\rm EL}.$ Furthermore, we denote the subgroup of $G(\Phi,R)$ generated by its root elements by $E(\Phi,R).$ It is called the \textit{elementary subgroup of $G(\Phi,R).$}

Next, we introduce bounded generation by root elements, the central concept of the paper:

\begin{definition}
Let $R$ be a commutative ring with $1$ and $\Phi$ an irreducible root system. Then $E(\Phi,R)$ is called \textit{boundedly generated by root elements}, if there is a natural number $N(\Phi,R)$ such that each element of $E(\Phi,R)$ can be written as a product of at most $N(\Phi,R)$ root elements. In this context, we define the following function:
\begin{align*}
|\cdot|_{EL}:E(\Phi,R)\to\mathbb{N}_0,X\mapsto
\begin{cases}
\min\{n\in\mathbb{N}\mid X\text{ is a product of }n\text{ root elements}\}&\text{ ,if }X\neq 1\\
0 &\text{ ,if }X=1
\end{cases}
\end{align*}
Set further $|E(\Phi,R)|_{EL}:=\max\{|X|_{EL}\mid X\in E(\Phi,R)\}.$
\end{definition}

\begin{remark}
We often also say that the group $G(\Phi,R)$ is \textit{boundedly generated by root elements}. This simply means that $G(\Phi,R)=E(\Phi,R)$ holds and $E(\Phi,R)$ is boundedly generated by root elements. Then we also use the notation $|G(\Phi,R)|_{EL}$ for $|E(\Phi,R)|_{EL}.$
\end{remark}

Additionally, the root elements associated to various distinct roots of $\Phi$ satisfy the commutator relations stated in the following lemma:

\begin{lemma}\cite[Proposition~33.2-33.5]{MR0396773}
\label{commutator_relations}
Let $\Phi$ be an irreducible root system and let be $R$ a commutative ring with $1$. Further, let $a,b\in R$ and $\alpha,\beta\in\Phi$ be given with $\alpha+\beta\neq 0$.
\begin{enumerate}
\item{If $\alpha+\beta\notin\Phi$, then $(\varepsilon_{\alpha}(a),\varepsilon_{\beta}(b))=1.$}
\item{If $\alpha,\beta$ are positive, simple roots in a root subsystem of $\Phi$ isomorphic to $A_2$, then\\ 
$(\varepsilon_{\beta}(b),\varepsilon_{\alpha}(a))=\varepsilon_{\alpha+\beta}(\pm ab).$}
\item{If $\alpha,\beta$ are positive, simple roots in a root subsystem of $\Phi$ isomorphic to $C_2$ with $\alpha$ short and $\beta$ long, then
\begin{align*}
&(\varepsilon_{\alpha+\beta}(b),\varepsilon_{\alpha}(a))=\varepsilon_{2\alpha+\beta}(\pm 2ab)\text{ and}\\
&(\varepsilon_{\beta}(b),\varepsilon_{\alpha}(a))=\varepsilon_{\alpha+\beta}(\pm ab)\varepsilon_{2\alpha+\beta}(\pm a^2b).
\end{align*}
}
\item{If $\Phi=G_2$ and $\alpha,\beta$ are positive simple roots with $\alpha$ short and $\beta$ long in $\Phi$, then 
\begin{align*}
&(\varepsilon_{\beta}(b),\varepsilon_{\alpha}(a))=\varepsilon_{\alpha+\beta}(\pm ab)\varepsilon_{2\alpha+\beta}(\pm a^2b)\varepsilon_{3\alpha+\beta}(\pm a^3b)
\varepsilon_{3\alpha+2\beta}(\pm a^3b^2),\\
&(\varepsilon_{\alpha+\beta}(b),\varepsilon_{\alpha}(a))=
\varepsilon_{2\alpha+\beta}(\pm 2ab)\varepsilon_{3\alpha+\beta}(\pm 3a^2b)\varepsilon_{3\alpha+2\beta}(\pm 3ab^2),\\
&(\varepsilon_{2\alpha+\beta}(b),\varepsilon_{\alpha}(a))=\varepsilon_{3\alpha+\beta}(\pm 3ab),\\
&(\varepsilon_{3\alpha+\beta}(b),\varepsilon_{\beta}(a))=\varepsilon_{3\alpha+\beta}(\pm ab)\text{ and}\\
&(\varepsilon_{2\alpha+\beta}(b),\varepsilon_{\alpha+\beta}(a))=\varepsilon_{3\alpha+2\beta}(\pm 3ab).
\end{align*}
}
\end{enumerate}
\end{lemma}

We also define certain congruence subgroups and some other notions that we need later on. 

\begin{definition}
Let $\Phi$ be an irreducible root system and let $R$ be a commutative ring with $1$ and $I$ an ideal of $R$. 
\begin{enumerate}
\item{The subgroup $E(\Phi,R,I)\subset G(\Phi,R)$ is defined as the normal subgroup of $G(\Phi,R)$ generated by the set $\{\varepsilon_{\phi}(x)\mid x\in I,\phi\in\Phi\}$.}
\item{The subgroup $C(\Phi,R,I)\subset G(\Phi,R)$ is defined as the kernel of the group homomorphism $\pi_I:G(\Phi,R)\to G(\Phi,R/I)$ if $I\neq R$ and as $C(\Phi,R,R):=G(\Phi,R)$ otherwise.}
\end{enumerate}
\end{definition}

\begin{remark}
As the root elements $\varepsilon_{\phi}(x)$ vanish under $\pi_I$ for $x\in I$ and $\phi\in\Phi,$ the inclusion $E(\Phi,R,I)\subset C(\Phi,R,I)$ holds.
\end{remark}

We also recall the word norms studied in Section~\ref{Strong_bound_section} and \ref{lower_bound_section}: 

\begin{definition}\label{conjugation_gen_word+basic_def}
Let $G$ be a group.
\begin{enumerate}
\item{We define $A^B:=BAB^{-1}$ and $(A,B):=ABA^{-1}B^{-1}$ for $A,B\in G$.}
\item{We denote two elements $A,B\in G$ being conjugate in $G$, by writing $A\sim B.$}
\item{For $T\subset G$, we define $\dl T\dr$ as the smallest normal subgroup of $G$ containing $T.$}
\item{A subset $T\subset G$ is called a \textit{normally generating set} of $G$, if $\dl T\dr=G$.}
\item{For $k\in\mathbb{N}$ and $T\subset G$ denote by 
\begin{equation*}
B_T(k):=\bigcup_{1\leq i\leq k}\{x_1\cdots x_i\mid\forall j\leq i: x_j\text{ conjugate to }A\text{ or } A^{-1}\text{ in }G\text{ and }A\in T\}\cup\{1\}.
\end{equation*}
Further set $B_T(0):=\{1\}.$ If $T$ only contains the single element $A$, then we write $B_A(k)$ instead of $B_{\{A\}}(k)$.}
\item{Define for a set $T\subset G$ the \textit{conjugation invariant word norm} $\|\cdot\|_T:G\to\mathbb{N}_0\cup\{+\infty\}$ by 
$\|A\|_T:=\min\{k\in\mathbb{N}_0|A\in B_T(k)\}$ for $A\in\dl T\dr$ and by $\|A\|_T:=+\infty$ for $A\notin\dl T\dr.$ The diameter 
$\|G\|_T={\rm diam}(\|\cdot\|_T)$ of $G$ is defined as the minimal $N\in\mathbb{N}$ such that $\|A\|_T\leq N$ for all $A\in G$ or as $+\infty$ if there is no such $N$.}
\item{Define for $l\in\mathbb{N}$ the invariant 
\begin{equation*}
\Delta_l(G):=\sup\{{\rm diam}(\|\cdot\|_T)|\ T\subset G\text{ with }|T|\leq l,\dl T\dr=G\}\in\mathbb{N}_0\cup\{\pm\infty\}
\end{equation*}
with $\Delta_k(G)$ defined as $-\infty$, if there is no normally generating set $T\subset G$ with $|T|\leq l.$ 
}
\item{The group $G$ is called \textit{strongly bounded}, if $\Delta_l(G)<+\infty$ for all $l\in\mathbb{N}$.}
\end{enumerate}
\end{definition}

Using this we also introduce the following 

\begin{definition}
Let $R$ be a commutative ring with $1$, $\Phi$ be an irreducible root system and $\phi\in\Phi.$ Further, let $T\subset G(\Phi,R)$ be given. For $k\in\mathbb{N}_0$ set $\varepsilon(T,\phi,k):=\{x\in R|\varepsilon_{\phi}(x)\in B_T(k)\}.$
\end{definition}

Before continuing, we will define the Weyl group elements in $G(\Phi,R)$:

\begin{definition}
Let $R$ be a commutative ring with $1$ and let $\Phi$ be a root system. Define for $\phi\in\Phi$ the element $w_{\phi}:=\varepsilon_{\phi}(1)\varepsilon_{-\phi}(-1)\varepsilon_{\phi}(1).$
\end{definition} 

Using these Weyl group elements, we can obtain the following lemma:

\begin{lemma}\cite[Lemma~20(b), Chapter~3, p.~23]{MR3616493}
Let $R$ be a commutative ring with $1$ and $\Phi$ an irreducible root system. Let $\phi\in\Phi,\alpha\in\Pi$ and $x\in R$ be given.
Then for each normally generating set $T$ of $G(\Phi,R)$ one has 
\begin{equation*}
\|\varepsilon_{\phi}(x)\|_T=\|\varepsilon_{w_{\alpha}(\phi)}(x)\|_T.
\end{equation*}   
Here, the element $w_{\alpha}(\phi)\in\Phi$ is defined by the action of the Weyl group $W(\Phi)$ on $\Phi.$ 
\end{lemma}

\section{Bounded generation by root elements in positive characteristic}
\label{proof_main}

In this section, we will prove Theorem~\ref{main_thm}. As is often the case for theorems concerning Chevalley groups $G(\Phi,R)$, the proof will split in various subcases depending on the root system $\Phi$ in question. More precisely, the proof will be done as follows: In the first subsection, we will use a Theorem by Morris \cite[Theorem~3.11]{MR2357719} to derive the finiteness of the universal Mennicken groups $M(I)$ for the ring of integers $R$ of a global function field $K$ and a non-zero ideal $I$ in $R$. Then in the second subsection, we will use this result together with a couple of additional first order properties to deduce that the groups $G(A_2,R)/E(A_2,R)$ and $G(C_2,R)/E(C_2,R)$ are finite. This will (almost) yield Theorem~\ref{main_thm} in case of $\Phi=A_2$ and $C_2$. Disregarding $\Phi=G_2,$ the general case can be reduced to those two cases by invoking a result due to Tavgen \cite[Proposition~1]{MR1044049}. This reduction will be the subject of the fourth subsection; in the fourth subsection we will also resolve the remaining case of $\Phi=G_2$. The key observation underlying all the arguments in this section is the following compactness theorem:

\begin{theorem}\cite[Theorem~2.8]{MR2357719}\label{compactness_bounded}
Let $\Phi$ be an irreducible root system and let $\C T$ be a first order theory of a language $\C L$ containing at least
\begin{enumerate}
\item{a predicate symbol $\C P$ and}
\item{$n_{\Phi}^2$ predicate symbols $\C Q_{i,j}$ for $1\leq i,j\leq n_{\Phi}$ where $n_{\Phi}$ is chosen as in Section~\ref{Chevalley}}
\end{enumerate}
Assume that $\C T$ has the following two properties:
\begin{enumerate}
\item{The universe $R$ of each model $\C M$ of $\C T$ is a commutative ring with $1$ such that $I:=\{x\in R\mid\C P^{\C M}(x)\}$ is a non-zero ideal of $R$.}
\item{For $R,I$ and $\C M$ as in $(1),$ the set $\{A=(a_{i,j})\in R^{n_{\Phi}\times n_{\Phi}}\mid\bigwedge_{1\leq i,j\leq n_{\Phi}}\C Q^{\C M}_{i,j}(a_{i,j})\}$ generates $E(\Phi,R,I),$}
\item{For $R$ and $I$ as in $(1)$, the set $C(\Phi,R,I)/E(\Phi,R,I)$ is finite.}
\end{enumerate}
Then for any universe $S$ of a model $\C N$ of $\C T$ with associated ideal $J:=\{x\in S\mid\C P^{\C M}(x)\}$, the set $D_J:=\{A=(a_{i,j})\in S^{n_{\Phi}\times n_{\Phi}}\mid \bigwedge_{1\leq i,j\leq n_{\Phi}}\C Q^{\C N}_{i,j}(a_{i,j})\}$ boundedly generates $E(\Phi,S,J)$. In particular, the maximal number $|E(\Phi,S,J)|_{D_J}$ of factors of $D_J$ needed to write any element of $E(\Phi,S,J)$ has an upper bound only depending on $\C T.$ 
\end{theorem}

\subsection{Bounding the cardinality of the universal Mennicken group} 

In order to apply Theorem~\ref{compactness_bounded}, we need a couple of first-order properties of rings defined by Morris \cite[Definition~3.2,3.6]{MR2357719}. However, in order to introduce them we need yet another definition:

\begin{definition}
Let $R$ be a commutative ring with $1$ and $a\in R$ given. Then the group $U(aR)$ is defined as the group of units of the ring $R/aR,$ if $a$ is not a unit in $R$ and as the trivial group if $a$ is a unit in $R.$ Furthermore, we define $e(aR)$ as the exponent ${\rm exp}(U(aR))$.
\end{definition}

\begin{remark}
Recall, that the exponent of a group $G$ is defined as the smallest $n\in\mathbb{N}$ such that $g^n=1_G$ holds for all $g\in G$ or as $+\infty$ if there is no such $n.$
\end{remark}

Next, we define the required first-order properties:

\begin{definition}\label{first_order_def}
Let $R$ be a commutative ring with $1$ and $t,r,m$ and $l$ be natural numbers. Then $R$ is said to  
\begin{enumerate}
\item{have \textit{(Bass) stable range} at most $m$, if for $a_0,a_1,\dots,a_m\in R$ with $a_0R+\cdots+a_mR=R,$ there are elements $x_1,\dots,x_m\in R$ such that 
$(a_1-x_1a_0)R+\cdots+(a_m-x_ma_0)=R.$}
\item{have \textit{stable range at most $3/2,$} if for each $a\in R-R^*,$ the ring $R/aR$ has stable range at most $1.$}
\item{satisfy the property ${\rm Gen}(t,r)$, if for all $a,b\in R$ with $aR+bR=R$, there is an element $h\in a+bR$ such that $U(hR)/U(hR)^t$ can be generated by $r$ or less elements.}
\item{satisfy the property ${\rm Exp}(t,l)$, if for any $q\in R-\{0\}$ and $(b,a)\in W(qR)$, there exist $a',c,d\in R$ as well as $u_i,f_i,g_i,b'_i,d'_i$ for $1\leq i\leq l$ such that the following properties hold:
\begin{enumerate}
\item{the element $a'$ is an element of $a+bR.$}
\item{The matrix 
\begin{equation*}
\begin{pmatrix}
a' & b\\
c & d
\end{pmatrix}
\end{equation*}
is an element of $C(A_1,R,qR).$
}
\item{The matrix 
\begin{equation*}
\begin{pmatrix}
a' & b'_i\\
c & d'_i
\end{pmatrix}
\end{equation*}
is an element of $C(A_1,R,qR)$ for all $1\leq i\leq l.$}
\item{The matrix 
\begin{equation*}
f_iI_2+g_i\cdot
\begin{pmatrix}
a' & b'_i\\
c & d'_i
\end{pmatrix}
\end{equation*}
is an element of $C(A_1,R,qR)$ for all $1\leq i\leq l.$}
\item{The product $(f_1+g_1\cdot a')^2\cdot(f_2+g_2\cdot a')^2\cdots(f_l+g_l\cdot a')^2$ is an element of $(a')^t+cR.$}
\item{The element $u_i$ is a unit in $R$ and $f_i+g_i\cdot a'\equiv u_i\text{ mod } b'_iR$ holds for all $1\leq i\leq l.$}
\end{enumerate}
}
\end{enumerate}
\end{definition}

Using these properties together with the aforementioned stable range conditions, Morris proves that

\begin{theorem}\label{finiteness_mennicken}\cite[Theorem~3.11]{MR2357719}
Let $R$ be a commutative ring with $1$ and $I$ a non-zero ideal of $R.$ Further, assume that $R$ is of stable range at most $3/2$ and satisfies ${\rm Gen}(t,r)$ and ${\rm Exp}(t,l)$ for positive integers $t,r$ and $l.$ Then the universal Mennicken group $M(I)$ is finite and its order is bounded above by $t^r.$
\end{theorem}

\begin{remark}
It is worth pointing out, that the fact that the universal Mennicken group $M(I)$ for $I$ a non-zero ideal in a ring of integers of global function fields, is finite is not the crux of the matter here; in fact said group $M(I)$ is trivial as shown by Bass, Milnor and Serre \cite[Theorem~3.6]{MR244257}. The point is that its finiteness can be derived from first order conditions alone, as this enables the application of the compactness argument, Theorem~\ref{compactness_bounded}. 
\end{remark}

Then he applied Theorem~\ref{finiteness_mennicken} to rings of algebraic integers. We will apply said theorem to rings of integers in global function fields. In order to do so however, we need the following version of Dirichlet's (or more historically accurate in this case Kornblum's) Theorem: 

\begin{theorem}\label{Dirichlet_positive_char}\cite[Theorem~(A.12)]{MR244257}
Let $K$ be a global function field that is a finite extension of $\mathbb{F}_q(T)$ and let $R:=\C O_K$ be the ring of integers in $K$ and $\mathbb{F}_q$ the field of constants of $K.$ Further, let $S_{\infty}$ be the primes covering the prime $\C P_{\infty}$ of $\mathbb{F}_q(T).$ Further, let $f,g\in\C O_K$ with $f\C O_K+g\C O_K=\C O_K$ and integers $n_{\C P}>0$ and $m_{\C P}$ be given for each $\C P\in S_{\infty}.$ Then there is an element $f'\in f+g\C O_K$ such that $f'\C O_K$ is a maximal ideal in $R$ and such that ${\rm ord}_{\C P}(f')\equiv m_{\C P}\text{ mod }n_{\C P}\mathbb{Z}$ holds for all $\C P\in S_{\infty}.$    
\end{theorem}

Using Theorem~\ref{Dirichlet_positive_char}, it is easy to deduce that a ring of integers $\C O_K$ in a global function field $K$ satisfies ${\rm Gen}(t,1):$

\begin{lemma}\label{gen_pos_char}\cite[Corollary~3.5]{MR2357719}
Let $K$ be a global function field with ring of integers $R:=\C O_K$. Then $R$ satisfies ${\rm Gen}(t,1)$ for every $t\in\mathbb{N}.$
\end{lemma}

\begin{proof}
Let $a,b\in R$ be given with $aR+bR=R.$ Then by Theorem~\ref{Dirichlet_positive_char}, there is an element $a'\in a+bR$ such that $R/a'R$ is a finite field. But then the unit group $U(a'R)$ is cyclic and so every quotient $U(a'R)/U(a'R)^t$ is also cyclic. This finishes the proof.  
\end{proof}

Deducing ${\rm Exp}(t,l)$ for the rings of integers at hand requires an intermediate step:

\begin{proposition}\label{positive_char_exp_prep}
Let $K$ be a global function field that is a finite extension of $\mathbb{F}_q(T)$ with ring of integers $R:=\C O_K$ and field of constants $\mathbb{F}_q.$ Also, set $k:=[K:\mathbb{F}_q(T)].$ Let $f,g\in R$ with $fR+gR=R, h\in R-\{0\}$ and $n\in\mathbb{N}$ be given. Then there is an element $f'\in R$ such that $f'\equiv f\text{ mod }gR, f'R+hR=R$ and such that $\gcd(e(f'R),n)$ divides $(q-1)!^{\lceil k\log_2(q)\rceil}.$
\end{proposition}

\begin{proof}
We may assume wlog that $g\neq 0$ and $gh\notin R^*.$ Further, note that $R$ is a ring of stable range at most $3/2$ and hence the ring $\bar{R}:=R/ghR$ has stable range $1.$ Additionally, $fR+gR=R$ implies $f\bar{R}+g\bar{R}=\bar{R}.$ Hence using the fact that $\bar{R}$ has stable range $1,$ we can find $x\in R$ such that $(f-xg)\bar{R}=\bar{R}.$ Thus $(f-xg)R+ghR=R$ follows. Hence we may also assume $g=h$ for the rest of the proof. Next, let $S$ be the set of rational primes dividing $n$ and for each $p\in S$ pick $r(p)\in\mathbb{N}_0$ maximal such that $p^{r(p)}$ divides $(q-1)!^{\lceil k\log_2(q)\rceil}.$ We assume that the primes of $K$ covering the prime at infinity $\C P_{\infty}$ of $\mathbb{F}_q(T)$ are $\mathfrak{P}_{\infty,1},\dots,\mathfrak{P}_{\infty,L}$ with relative degrees $f_j:=f(\mathfrak{P}_{\infty,j}\mid \C P_{\infty})$ and ramification indices $e_j:=e(\mathfrak{P}_{\infty,j}\mid \C P_{\infty})$. Further, set $t_p:={\rm exp}(\mathbb{Z}_{p^{r(p)+1}}^*)$ for each rational prime $p\in S.$ Then by Theorem~\ref{Dirichlet_positive_char}, we can find an element $f'\in R$ such that $f'\equiv f\text{ mod }gR, f'R$ is a maximal ideal and such that
\begin{equation*}
{\rm ord}_{\mathfrak{P}_{\infty,j}}(f'R)\equiv -e_j\text{ mod }{\rm lcm}\{t_p\mid p\text{ prime and }p\in S\}\mathbb{Z}
\end{equation*}
holds for all $j=1,\dots,L.$ To finish the proof, we only have to show that $\gcd(e(f'R),n)$ divides $(q-1)!^{\lceil k\log_2(q)\rceil}$. To this end, assume for contradiction that $\gcd(e(f'R),n)$ does not divide $(q-1)!^{\lceil k\log_2(q)\rceil}.$ Then there must be a prime $p\in S$ such that $p^{r(p)+1}$ divides $e(f'R).$ But note that as $f'R$ is a maximal ideal, we obtain that $R/f'R$ is a finite field and hence the multiplicative group $(R/f'R)^*$ is cyclic of order $q^{{\rm deg}(f'R)}-1.$ This implies by Lemma~\ref{degree_func} that
\begin{align*}
1\equiv q^{{\rm deg}(f'R)}=q^{-\sum_{j=1}^L f_j\cdot{\rm ord}_{\mathfrak{P}_{\infty,j}}(f'R)}=\prod_{j=1}^L (q^{-{\rm ord}_{\mathfrak{P}_{\infty,j}}(f'R)\cdot f_j})\equiv\prod_{j=1}^L q^{e_j\cdot f_j}=q^{\sum_{j=1}^L e_j\cdot f_j}=q^k\text{ mod }p^{r(p)+1}\mathbb{Z}. 
\end{align*}
The last equality follows from Theorem~\ref{sum_formula}. But now we distinguish two cases: First, assume that $p\geq q^k.$ But then $1\equiv q^k\text{ mod }p^{r(p)+1}\mathbb{Z}$ is clearly impossible, as $1<q^k<p^{r(p)+1}$ holds. Second, we assume $p<q^k.$ Then $r(p)\geq k\log_2(q)\geq k\log_p(q)=\log_p(q^k)$ holds by definition and hence $p^{r(p)+1}\geq p^{\log_p(q^k)+1}=q^k\cdot p>q^k.$ But again $1<q^k<p^{r(p)+1}$ implies that $1\equiv q^k\text{ mod }p^{r(p)+1}\mathbb{Z}$ is impossible. This finishes the proof.
\end{proof}

This proposition in hand, the proof of the property ${\rm Exp}(t,l)$ is now straightforward:

\begin{proposition}\label{exp_positive_char}
Let $K$ be a global function field that is a finite extension of $\mathbb{F}_q(T)$ and field of constants $\mathbb{F}_q.$ Also, set $k:=[K:\mathbb{F}_q(T)].$ Then $R$ satisfies ${\rm Exp}(2\cdot(q-1)!^{\lceil k\log_2(q)\rceil},2).$
\end{proposition}

\begin{proof}
The proof works the same way as the proof of \cite[Theorem~3.9]{MR2357719}. The only difference is that rather than using \cite[Lemma~3.8(2)]{MR2357719}, we use Proposition~\ref{positive_char_exp_prep}.
\end{proof}

\begin{remark}
The bound $2\cdot(q-1)!^{\lceil k\log_2(q)\rceil}$ above can be lowered in specific cases. For example, one can replace $k$ with the greatest common divisor of the degrees $f_j$ of the primes covering the prime at infinity of $\mathbb{F}_q(T)$.
\end{remark}

\subsection{Deducing bounded generation in rank $2$}

\subsubsection{The root system $\Phi=A_2$}

We need two technical statements before we can prove the main theorem of this subsection. First note that for a given commutative ring $R$, the group ${\rm SL}_2(R,I):=C(A_1,R,I)$ is considered to be a subgroup of ${\rm SL}_3(R,I):=C(A_2,R,I)$ by way of the monomorphism
\begin{equation*}
\phi_{\beta}:{\rm SL}_2(R)\to{\rm SL}_3(R),A\mapsto
\begin{pmatrix}
A & \ \\
\ & 1
\end{pmatrix}
\end{equation*}
 
Next, the first technical statement needed is the following:

\begin{proposition}\cite[Theorem~2.14]{MR2357719}\label{normality_a2}
Let $R$ be a commutative ring with $1$ and $I$ a non-zero ideal of $R.$ Further, assume that $R$ has stable range at most $2$ and that $n\geq 3$ is given. Then 
\begin{enumerate}
\item{The group $E(A_2,R)$ is normal in $G(A_2,R).$}
\item{The equation ${\rm SL}_3(R,I)={\rm SL}_2(R,I)\cdot E(A_2,R,I)$ holds.}
\end{enumerate}
\end{proposition}


Now, we can state the main theorem of this subsection:

\begin{theorem}\label{A_2_bounded_gen}
Let $K$ be a global function field that is a finite extension of $\mathbb{F}_q(T)$ with ring of integers $R:=\C O_K$ and field of constants $\mathbb{F}_q.$ Also, set $k:=[K:\mathbb{F}_q(T)].$ Further, let $I$ be a non-zero ideal of $R.$
\begin{enumerate}
\item{Then the subgroup $E(A_2,R,I)$ can be boundedly generated by the set $E_I:=\{A\varepsilon_{\phi}(x)A^{-1}\mid A\in G(A_2,R),x\in I,\phi\in A_2\}.$ More precisely, there is a constant $S(q,k,A_2)\in\mathbb{N}$ depending only on $q$ and $k$ such that $\|E(A_2,R,I)\|_{\rm E_I}\leq S(q,k,A_2).$}
\item{Then $E(A_2,R)$ is boundedly generated by root elements. More precisely, there is a constant $L(q,k,A_2)\in\mathbb{N}$ depending only on $q$ and $k$ such that with $|E(A_2,R)|_{\rm EL}\leq L(q,k,A_2).$}
\end{enumerate}
\end{theorem}

The second technical statement, we need is the following:

\begin{proposition}\cite[Theorem~2.18]{MR2357719}\label{mennicken_a2}
Let $R$ be a commutative ring with $1$ and $I$ a non-zero ideal of $R.$ Consider the map
\begin{align*}
&\{\cdot\}_I:W(I)\to C(A_2,R,I)/E(A_2,R,I),(b,a)\mapsto 
\phi_{\beta}
\begin{pmatrix}
a & b\\
* & *
\end{pmatrix}
\cdot E(A_2,R,I).
\end{align*}
Then this map is well-defined and is a Mennicken symbol.
\end{proposition}

Now, we can prove Theorem~\ref{A_2_bounded_gen}:

\begin{proof}
According to Theorem~\ref{finiteness_mennicken}, for any commutative ring $S$ with $1$ and a non-zero ideal $I$ such that $S$ has stable range at most $3/2$ and satisfies ${\rm Exp}(2\cdot(q-1)!^{\lceil k\log_2(q)\rceil},2)$ and ${\rm Gen}(2\cdot(q-1)!^{\lceil k\log_2(q)\rceil},1)$, the universal Mennicken group $M(I)$ has less than $2\cdot(q-1)!^{\lceil k\log_2(q)\rceil}$ many elements. But according to Proposition~\ref{mennicken_a2}, the map 
\begin{align*}
&\{\cdot\}_I:W(I)\to C(A_2,S,I)/E(A_2,S,I),(b,a)\mapsto 
\phi_{\beta}
\begin{pmatrix}
a & b\\
* & *
\end{pmatrix}
\cdot E(A_2,S,I).
\end{align*}
is a Mennicken symbol. Thus by Proposition~\ref{universal_mennicken}, there is a group homomorphism 
\begin{align*}
\theta_I:M(I)\to C(A_2,S,I)/E(A_2,S,I),\bracenom{b}{a}_I\mapsto
\phi_{\beta}\begin{pmatrix}
a & b\\
* & *
\end{pmatrix}
\cdot E(A_2,S,I)
\end{align*}
But according to Proposition~\ref{normality_a2}(2), the map $\theta_I$ is surjective and hence $C(A_2,S,I)/E(A_2,S,I)$ is a quotient of $M(I)$. Thus $C(A_2,S,I)/E(A_2,S,I)$ has less than $2\cdot(q-1)!^{\lceil k\log_2(q)\rceil}$ many elements. But we have only used the first order properties ${\rm Exp}(2\cdot(q-1)!^{\lceil k\log_2(q)\rceil},2)$ and ${\rm Gen}(2\cdot(q-1)!^{\lceil k\log_2(q)\rceil},1)$ and stable range at most $3/2$ to deduce this; so we can use Theorem~\ref{compactness_bounded} to deduce that the set $E_I:=\{A\varepsilon_{\phi}(x)A^{-1}\mid A\in G(A_2,S),x\in I\}$ boundedly generates the group $E(A_2,S,I)$ with a bound $S(\C T_{q,k})$ on $\|E(A_2,S,I)\|_{E_I}$ only depending on a first order theory $\C T_{q,k}$ of commutative rings with $1$ containing the first order statements ${\rm Exp}(2\cdot(q-1)!^{\lceil k\log_2(q)\rceil},2), {\rm Gen}(2\cdot(q-1)!^{\lceil k\log_2(q)\rceil},1)$ and stable range at most $3/2.$ Phrased differently, the bound $S(q,k,A_2):=S(\C T_{q,k})$ only depends on $q$ and $k.$ But according to Proposition~\ref{exp_positive_char}, the ring $R=\C O_K$ satisfies ${\rm Exp}(2\cdot(q-1)!^{\lceil k\log_2(q)\rceil},2)$ and ${\rm Gen}(2\cdot(q-1)!^{\lceil k\log_2(q)\rceil},1)$ according to Lemma~\ref{gen_pos_char} and $R$ has stable range at most $3/2$. This finishes the proof of the first claim of the theorem.

For the second claim, note first that by Proposition~\ref{normality_a2}, the subgroup $E(A_2,R)$ is normal in $G(A_2,R).$ Hence in particular, $E(A_2,R,R)=E(A_2,R)$ and $C(A_2,R,R)=G(A_2,R)$ hold. Phrased differently, the last claim of the theorem follows in essentially the same manner as the first one by considering the set ${\rm EL}=\{\varepsilon_{\phi}(x)\mid x\in R,\phi\in A_2\}$ instead of the set $E_I.$ 
\end{proof}

\subsubsection{The root system $\Phi=C_2$}

The case of $\Phi=C_2$ is arguably more involved than the case of $\Phi=A_2$. We need to describe a specific subgroup of ${\rm Sp}_4(R)=G(C_2,R)$ first. As mentioned in Section~\ref{Chevalley}, we consider ${\rm Sp}_4(R)$ as the following subgroup of ${\rm GL}_4(R):$ 
\begin{equation*}
{\rm Sp}_4(R)=
\{A\in{\rm GL}_4(R)\mid A^T\cdot 
\left(\begin{array}{c|c}
0_2 & I_2\\
\midrule
-I_2 & 0_2
\end{array}
\right)\cdot A=
\left(\begin{array}{c|c}
0_2 & I_2\\
\midrule
-I_2 & 0_2
\end{array}
\right)\}
\end{equation*}
For a commutative ring $R$ and an ideal $I$ of $R,$ we leave it as an exercise to the reader to check that the following set $G_{\beta}(R,I)$ is a subgroup of $C(C_2,R,I):$

\begin{equation*}
G_{\beta}(R,I)=
\{ 
\left(\begin{array}{c|c}
\begin{matrix}
a & 0\\
0 & 1
\end{matrix}
& 
\begin{matrix}
b & 0\\
0 & 0
\end{matrix}\\
\midrule
\begin{matrix}
c & 0\\
0 & 0
\end{matrix} 
& 
\begin{matrix}
d & 0\\
0 & 1
\end{matrix}
\end{array}
\right)\mid 
\begin{pmatrix}
a & b\\
c & d
\end{pmatrix}\in C(A_1,R,I)={\rm SL}_2(R,I)\}
\end{equation*}

In this context, we also consider the map
\begin{equation*}
\phi_{\beta}:C(A_1,R,I)\to G_{\beta}(R,I),
\begin{pmatrix}
a & b\\
c & d
\end{pmatrix}\mapsto
\left(\begin{array}{c|c}
\begin{matrix}
a & 0\\
0 & 1
\end{matrix}
& 
\begin{matrix}
b & 0\\
0 & 0
\end{matrix}\\
\midrule
\begin{matrix}
c & 0\\
0 & 0
\end{matrix} 
& 
\begin{matrix}
d & 0\\
0 & 1
\end{matrix}
\end{array}\right)
\end{equation*}
Similar to the $A_2$-case, we also note the following: 

\begin{proposition}\cite[Proposition~13.2]{MR244257}\cite[Theorem~1]{MR843808}\label{sp_4_decomposition}
Let $R$ be a commutative ring of stable range at most $3/2$ and $I$ an ideal of $R.$ Then 
\begin{enumerate}
\item{$E(C_2,R)$ is a normal subgroup of $G(C_2,R)$ and}
\item{$C(C_2,R,I)=G_{\beta}(R,I)\cdot E(C_2,R,I)$ holds.}
\end{enumerate}
\end{proposition}

The main technical step of this section is to derive the following theorem:

\begin{theorem}\label{Finiteness_congruence_quotient_symplectic}
Let $R$ be a commutative ring with $1$ and $I$ a non-zero ideal of $R$. Further, assume that $R$ is of stable range at most $3/2$ and satisfies ${\rm Gen}(2,1),{\rm Gen}(t,1)$ and ${\rm Exp}(t,l)$ for positive integers $t$ and $l.$ Then $C(C_2,R,I)/E(C_2,R,I)$ is finite and its order is bounded above by $2t.$ 
\end{theorem}

This theorem will be proven by deriving first that a certain subgroup $G(I)$ of $C(C_2,R,I)/E(C_2,R,I)$ is a quotient of the universal Mennicken group $M(I)$ and so $G(I)$ will have at most $t$ many elements according to Theorem~\ref{finiteness_mennicken}. Then we will finish the proof by showing that $G(I)$ has index at most $2$ in $C(C_2,R,I)/E(C_2,R,I)$. First, we need the following proposition collecting a couple of observations from \cite{MR244257}:
 
\begin{proposition}\label{square_mennicken}
Let $R$ be a commutative ring with $1$ and $I$ a non-zero ideal in $R.$ Further, set $Q(I):=C(C_2,R,I)/E(C_2,R,I).$ Then the two maps 
\begin{align*}
&\{\cdot\}_I:W(I)\to Q(I),(b,a)\mapsto 
\phi_{\beta}
\begin{pmatrix}
a & b\\
* & *
\end{pmatrix}
\cdot E(C_2,R,I)=:\{b,a\}_I\text{ and }\\
&[\cdot]_I:W(I)\to Q(I),(b,a)\mapsto\{b^2,a\}_I:=[b,a]_I
\end{align*}
are well-defined and both satisfy ${\rm MS}(1).$ Furthermore, for $(b_1,a),(b_2,a)\in W(I),$ the equation $[b_1,a]_I\cdot\{b_2,a\}_I=\{b_1^2\cdot b_2,a\}_I$ holds. 
This implies in particular that $[\cdot]_I:W(I)\to Q(I)$ is a Mennicken symbol.
\end{proposition}

\begin{proof}
It follows from \cite[Lemma~5.5]{MR244257} and \cite[Theorem~1]{MR843808} that both functions are well-defined and satisfy the property ${\rm MS}(1).$ Then the fact that the equation $[b_1,a]_I\cdot\{b_2,a\}_I=\{b_1^2\cdot b_2,a\}_I$ holds for all $(b_1,a),(b_2,a)\in W(I)$ is precisely the content of \cite[Lemma~13.3]{MR244257}.
\end{proof}

Next, we define $G(I)$ as the subgroup of $C(C_2,R,I)/E(C_2,R,I)$ generated by the image of the map $[\cdot]_I.$ But then according to Proposition~\ref{square_mennicken}, the map $[\cdot]_I:W_I\to G(I)$ is a Mennicken symbol and so according to Theorem~\ref{finiteness_mennicken} the group $G(I)$ is finite and its order is bounded above by $t.$ Last, we will prove the following proposition:

\begin{proposition}\label{technical_square_prop}
Let $R$ be a commutative ring with $1$ and $I$ a non-zero ideal in $R.$ Assume further that $R$ has stable range at most $3/2$ and satisfies ${\rm Gen}(2,1).$ Last, let $D$ be an abelian group and $f:W(I)\to D$ be a function with the following properties:
\begin{enumerate}
\item{The function $f$ satisfies ${\rm MS}(1)$.}
\item{The map $f$ is surjective.}
\item{The function $g:W(I)\to D,(b,a)\mapsto f(b^2,a)$ satisfies ${\rm MS}(1)$.}
\item{The equation $f(b_1^2,a)\cdot f(b_2,a)=f(b_1^2\cdot b_2,a)$ holds for all $(b_1,a),(b_2,a)\in W_I.$}
\end{enumerate}
Then the subgroup $D'$ of $D$ generated by the image of $g$ in $D$ has index at most $2$ in $D.$ 
\end{proposition} 

This result is an adaption of an argument found at the end of the the proof of \cite[Theorem~3.7]{MR244257} to a first order setup:

\begin{proof}
Assume for contradiction that $[D:D']\geq 3.$ Then by the surjectivity of $f,$ we can choose $(b_1,a_1),(b_2,a_2),(b_3,a_3)\in W(I)$ such that $f(b_1,a_1)D',f(b_2,a_2)D'$ and $f(b_3,a_3)D'$ are three distinct cosets of $D'$ in $D.$ Next pick $d\in I$ non-zero and set $J:=dR$. Then one can note that \cite[Lemma~2.3]{MR244257} merely uses that proper quotients of $R$ have stable range $1$ (rather than them being semi-local as stated.) Hence it follows that each element $(b_i,a_i)$ is $J$-equivalent to an element $(b'_i,a'_i)\in W(J).$ But then one can use \cite[Lemma~2.4]{MR244257} to find elements $a,c_1,c_2,c_3\in R$ such that each $(b'_i,a'_i)$ is $J$-equivalent to $(c_i\cdot d,a)\in W(J).$ However, note that $aR+c_1c_2c_3dR=R.$ Thus invoking ${\rm Gen}(2,1)$, we can find $a'\equiv a\text{ mod }c_1c_2c_3dR$ such that $U(a'R)/U(a'R)^2$ is cyclic. To simplify notation assume that $a'=a$. Next, set $U:=U(aR).$ Consider the map
\begin{equation*}
h:U\to D,x+aR\mapsto f(x\cdot d,a).
\end{equation*}
This map is well-defined due to the choice of $a$ and the fact that $f$ satisfies ${\rm MS}(1).$ Next, for any $w\in U,$ consider the map $h_w:U\to D,u\mapsto h(w\cdot u).$ We want to note that $h_w(w^{-1}u)=h(u)$ holds for all $w,u\in U$ and hence the image of $h_w$ and the image of $h$ in $D$ agree. Next, choose a $d'\in R$ such that $d'd\equiv 1\text{ mod } aR$ holds and define $w\in U$ as $w:=(d')^2\cdot d+aR\in U.$ Then observe for $u+aR,v+aR\in U$ that
\begin{align*}
h_w(v^2+aR)\cdot h_w(u+aR)=f(v^2\cdot(d'd)^2,a)\cdot f(u\cdot(d'd)^2,a)=f(v^2u\cdot(d'd)^4,a)
\end{align*} 
holds due to the fourth property assumed for $f$. But further 
\begin{equation*}
v^2u(d'd)^4-v^2u(d'd)^2=v^2u[(d'd)^2-1]\cdot (d'd)^2\in aI
\end{equation*}
holds by choice of $d'$ and so $(v^2u\cdot(d'd)^4,a)$ and $(v^2u\cdot(d'd)^2,a)$ are $I$-equivalent. Thus $f(v^2u\cdot(d'd)^4,a)=f(v^2u\cdot (d'd)^2,a)=h_w(v^2u)$ holds as $f$ satisfies ${\rm MS}(1)$. Hence we have shown 
\begin{equation}\label{almost_hom_eq}
h_w(v^2)\cdot h_w(u)=h_w(v^2\cdot u)
\end{equation}
for all $u,v\in U.$ Next, choose $b\in R$ such that $b+aR$ generates $U/U^2.$ Hence each element $x+aR$ of $U$ has the form $x+aR=b^k\cdot v^2+aR$ for $v+aR\in U(aR)$ and $k\in\{0,1\}.$ Thus invoking (\ref{almost_hom_eq}), we obtain $h_w(x+aR)=h_w(b^k+aR)\cdot h_w(v^2+aR).$ However, $h_w(v^2+aR)=f(v^2\cdot (d'd)^2,a)$ is contained in $D'$. Finally distinguishing, the two cases of $k$, one can see that $h_w(x+aR)$ is an element of the set $(h_w(b+aR)\cdot D')\cup D'.$ So phrased differently, the image of $h_w$ in $D$ is contained in $(h_w(b+aR)\cdot D')\cup D'$. Hence the image of $h_w$ is contained in at most $2$ cosets of $D'$ in $D.$ But as remarked before, the image of $h_w$ agrees with the image of $h.$ But clearly the image of $h$ contains all three elements $f(c_1\cdot d,a)=h(c_1+aR),f(c_2\cdot d,a)=h(c_2+aR)$ and $f(c_3\cdot d,a)=h(c_3+aR)$ and so the image of $h$ intersects with three distinct cosets of $D'$ in $D.$ This contradiction finishes the proof.
\end{proof}

\begin{remark}
Assuming the weaker condition ${\rm Gen}(2,r)$ for $r\geq 1$ instead of ${\rm Gen}(2,1)$, one could derive with essentially the same argument as above that $D'$ has index at most $2^r$ in $D,$ but this sharper result is not necessary for our investigation.
\end{remark}

Finally, Theorem~\ref{Finiteness_congruence_quotient_symplectic} is shown as follows:

\begin{proof}
According to Proposition~\ref{sp_4_decomposition}(2), the map ${\rm SL}_2(R,I)\to C(C_2,R,I)/E(C_2,R,I),A\mapsto\phi_{\beta}(A)$ is surjective. Thus the map $\{\cdot\}_I:W(I)\to C(C_2,R,I)/E(C_2,R,I)$ from Proposition~\ref{square_mennicken} is surjective as well. Furthermore, the claims in Proposition~\ref{square_mennicken} imply that the map $\{\cdot\}_I:W(I)\to C(C_2,R,I)/E(C_2,R,I)$ satisfies all the properties needed from $f$ in Proposition~\ref{technical_square_prop}. Thus $G(I)=\langle\{b^2,a\}_I\mid (b,a)\in W(I)\rangle$ has index $2$ in $C(C_2,R,I)/E(C_2,R,I).$ However, as seen before Theorem~\ref{finiteness_mennicken} implies that $G(I)$ has at most $t$ elements. So in summary, $C(C_2,R,I)/E(C_2,R,I)$ has at most $2t$ elements. This finishes the proof.
\end{proof} 

Last, we derive the main result of this subsection:

\begin{theorem}\label{bound_gen_c2}
Let $K$ be a global function field that is a finite extension of $\mathbb{F}_q(T)$ with ring of integers $R:=\C O_K$ and field of constants $\mathbb{F}_q.$ Also, set $k:=[K:\mathbb{F}_q(T)].$ Further, let $I$ be a non-zero ideal of $R.$
\begin{enumerate}
\item{Then the subgroup $E(C_2,R,I)$ can be boundedly generated by the set $E_I:=\{A\varepsilon_{\phi}(x)A^{-1}\mid A\in G(C_2,R),x\in I,\phi\in C_2\}.$ More precisely, there is a constant $S(q,k,C_2)\in\mathbb{N}$ depending only on $q$ and $k$ such that $\|E(C_2,R,I)\|_{\rm E_I}\leq S(q,k,C_2).$}
\item{Then $E(C_2,R)$ is boundedly generated by root elements. More precisely, there is a constant $L(q,k,C_2)\in\mathbb{N}$ depending only on $q$ and $k$ such that with $|E(C_2,R)|_{\rm EL}\leq L(q,k,C_2).$}
\end{enumerate}
\end{theorem}

\begin{proof}
The proof works essentially the same way as the proof of of Theorem~\ref{A_2_bounded_gen}. The only difference is the use of Theorem~\ref{Finiteness_congruence_quotient_symplectic} to deduce the finiteness of $C(C_2,R,I)/E(C_2,R,I)$ and $G(C_2,R)/E(C_2,R)$ respectively from the conditions ${\rm Exp}(2\cdot(q-1)!^{\lceil k\log_2(q)\rceil},2),{\rm Gen}(2\cdot(q-1)!^{\lceil k\log_2(q)\rceil},1)$ and ${\rm Gen}(2,1)$ satisfied by $R.$ This is then used together with Theorem~\ref{compactness_bounded} to finish the proof.
\end{proof}

\begin{remark}
The proof presented here to derive bounded generation by root elements for $E(C_2,R)$ in case $R$ is the ring of integers in a global function field applies equally well to the case that $R$ is the ring of integers in an algebraic number field as the only properties of $R$ relevant to the proof are satisfied in both cases.
\end{remark}

\subsection{Finishing the proof of Theorem~\ref{main_thm}}

Finishing the proof of Theorem~\ref{main_thm} is facilitated by the following rank reduction step due to Tavgen:

\begin{proposition}\cite[Proposition~1]{MR1044049}\label{Tavgen_reduction}
Let $R$ be a commutative ring with $1$ such that there is a $L(R)\in\mathbb{N}$ such that for each irreducible root system $\Phi_0$ of rank $2$ except $G_2,$ one has
$E(\Phi_0,R)=(U^+(\Phi_0,R)\cdot U^-(\Phi_0,R))^{L(R)}.$ Then $E(\Phi,R)=(U^+(\Phi,R)\cdot U^-(\Phi,R))^{L(R)}$ holds for all irreducible root systems $\Phi$ of rank at least $2$ except $G_2.$ 
\end{proposition} 

To finish the proof of Theorem~\ref{main_thm} note first that by Theorem~\ref{A_2_bounded_gen}(2) and Theorem~\ref{bound_gen_c2}(2), there is a constant $L'(q,k)$ only depending on $q$ and $k$ such that 
\begin{align*}
E(\Phi_0,R)=(U^+(\Phi_0,R)\cdot U^-(\Phi_0,R))^{L'(q,k)}
\end{align*}
holds for $\Phi=A_2$ or $C_2$. But as these two root systems are the only irreducible root systems of rank $2$ besides $G_2$, one obtains by invoking Proposition~\ref{Tavgen_reduction} that 
\begin{align*}
E(\Phi,R)=(U^+(\Phi,R)\cdot U^-(\Phi,R))^{L'(q,k)}
\end{align*}
holds for all irreducible root systems of rank at least $2$ besides $G_2.$ But using Matsumoto's \cite[Corollaire~4.6]{MR0240214} together with the triviality of the universal Mennicken group $M(R)$ as established by Bass, Milnor and Serre for rings of integers $R$ of global function fields \cite[Theorem~3.6]{MR244257}, we have $G(\Phi,R)=E(\Phi,R)$ and so the proof is finished excluding the case $\Phi=G_2$. However, the case of $G_2$ is straightforward; the crucial observation here is the following lemma due to Tavgen, which finishes the proof of Theorem~\ref{main_thm} simply by applying the already shown $A_2$-case of it:

\begin{lemma}\cite[Proposition~5]{MR1044049}
Let $K$ be a global function field that is a finite extension of $\mathbb{F}_q(T)$ with field of constants $\mathbb{F}_q$ and ring of integers $R:=\C O_K$. Then each element $A$ of $G_2(R)$ can by multiplication with at most $18$ root elements of $G_2(R)$ be transformed into an element of the subgroup ${\rm SL}_3(R)$ of $G_2(R)$ generated by the set $\{\varepsilon_{\phi}(x)\mid x\in R,\phi\in G_2\text{ long }\}.$ 
\end{lemma}

\begin{proof}
One merely needs to note that the proof of \cite[Proposition~5]{MR1044049} only uses the Chinese Remainder Theorem and the fact that $R$ is a Dedekind domain rather than the ring of S-algebraic integers in a number field, so the proof still valid in the current setup.
\end{proof}

\section{Strong boundedness in positive characteristic}\label{Strong_bound_section}

The main subject of this section is to prove Theorem~\ref{strong_bound_pos}. In all but two cases, this is straightforward by invoking the following theorem:

\begin{theorem}\cite[Theorem~3.1]{General_strong_bound}\label{strong_bound_original}
Let $\Phi$ be an irreducible root system of rank at least $2$ and let $R$ be a commutative ring with $1$. Additionally, let $G(\Phi,R)$ be boundedly generated by root elements and if $\Phi=C_2$ or $G_2$, then we further assume $(R:2R)<\infty.$ Then there is a constant $C(\Phi,R)\in\mathbb{N}$ such that $\Delta_l(G(\Phi,R))\leq C(\Phi,R)\cdot l$ holds for all $l\in\mathbb{N}.$
\end{theorem}

As seen by inspecting the proof of Theorem~\ref{strong_bound_original}, the bound $C(\Phi,R)$ only depends on the ring $R$ by way of the bounds for bounded generation by root elements for $G(\Phi,R)$, if $\Phi\neq C_2$ and $G_2$. If $\Phi=C_2$ or $G_2$ and $2$ is a unit in $R,$ then $C(\Phi,R)$ also does not depend on $R$ beyond said bounds. Hence, if the characteristic of the global function field $K$ in question is not $2,$ then Theorem~\ref{strong_bound_pos} follows directly from Theorem~\ref{strong_bound_original} and Theorem~\ref{main_thm}. Thus the only cases left are that $\Phi=C_2$ or $G_2$ and that the global function field has characteristic $2.$

\subsection{The case of $C_2$}

Recall that the root system $C_2$ has eight roots and we can choose psoitive, simple roots $\alpha,\beta\in C_2$ such that $\alpha$ is short and $\beta$ is long and the set of positive roots given by the choice of $\alpha$ and $\beta$ as positive, simple roots are $C_2^+=\{\alpha,\beta,\alpha+\beta,2\alpha+\beta\}.$ For the purpose of this and the next subsection, we also advise the reader to recall the commutator formulas from Lemma~\ref{commutator_relations}.

In the proof of the ${\rm Sp}_4$-case of Theorem~\ref{strong_bound_original}, we demonstrate how to find the set $\{\varepsilon_{\phi}(2x)\mid x\in R,\phi\in C_2\}$ in a small $S$-ball for $S$ a set normally generating ${\rm Sp}_4(R)$ and then finish the proof using bounded generation together with $2R$ having finite index in $R.$ The argument for a global function field of characteristic $2$ is done in a similar manner, but with the ideal $2R$ replaced by a different ideal, the so-called \textit{booleanizing ideal ${\rm vn}_2(R)$} of $R$. 

\begin{proposition}\label{char2_uniform_boundedness_tech_prop}
Let $K$ be a global function field that is a finite extension of $\mathbb{F}_q(T)$ and field of constants $\mathbb{F}_q$ with $q$ a power of $2.$ Further, let $R:=\C O_K$ be the ring of integers of $K$ and let $N$ be the normal subgroup of ${\rm Sp}_4(R)$ generated by 
\begin{equation*}
A:=\varepsilon_{\alpha+\beta}(1)\varepsilon_{2\alpha+\beta}(1).
\end{equation*}
Also let $\|\cdot\|_A:{\rm Sp}_4(R)\to\mathbb{N}_0\cup\{+\infty\}$ be the conjugation generated word norm on ${\rm Sp}_4(R)$ defined as in Defintion~\ref{conjugation_gen_word+basic_def}. Then 
\begin{enumerate}
\item{$N$ is a finite index subgroup of ${\rm Sp}_4(R)$ and}
\item{the norm $\|\cdot\|_A$ has finite diameter on $N$.}
\end{enumerate}
\end{proposition}

\begin{proof}
First, we will show that 
\begin{equation*}
I:={\rm vn}_2(R):=(x-x^2|x\in R)\subset\varepsilon(A,\phi,18k)
\end{equation*}
for all $\phi\in C_2$ and second, we will deduce the two statements of the proposition from this.
First, observe for any $x\in R$ that
\begin{equation}
\label{char2_uniform_boundedness_eq0}
B_A(2)\ni(\varepsilon_{\alpha+\beta}(1)\varepsilon_{2\alpha+\beta}(1),\varepsilon_{-\beta}(x))=\varepsilon_{\alpha}(x)\varepsilon_{2\alpha+\beta}(x).
\end{equation}
This yields that
\begin{align*}
B_A(2)\ni &w_{\beta}w_{\alpha}w_{\beta}\varepsilon_{\alpha}(x)\varepsilon_{2\alpha+\beta}(x)w_{\beta}^{-1}w_{\alpha}^{-1}w_{\beta}^{-1}\\
&=w_{\beta}w_{\alpha}\varepsilon_{\alpha+\beta}(x)\varepsilon_{2\alpha+\beta}(x)w_{\alpha}^{-1}w_{\beta}^{-1}\\
&=w_{\beta}\varepsilon_{\alpha+\beta}(x)\varepsilon_{\beta}(x)w_{\beta}^{-1}=\varepsilon_{\alpha}(x)\varepsilon_{-\beta}(x).
\end{align*}
On the other hand (\ref{char2_uniform_boundedness_eq0}) implies for $x,y\in R$ that 
\begin{align*}
B_A(4)\ni(\varepsilon_{\alpha}(y)\varepsilon_{2\alpha+\beta}(y),\varepsilon_{-(\alpha+\beta)}(x))
&=(\varepsilon_{2\alpha+\beta}(y),\varepsilon_{-(\alpha+\beta)}(x))^{\varepsilon_{\alpha}(y)}\varepsilon_{-\beta}(2xy)\\
&\sim\varepsilon_{\alpha}(xy)\varepsilon_{-\beta}(x^2y).
\end{align*}
But this implies together with $\varepsilon_{\alpha}(x)\varepsilon_{-\beta}(x)\in B_A(2)$ that 
\begin{equation*}
\varepsilon_{-\beta}(x^2y-xy)=\varepsilon_{-\beta}(x^2y)\varepsilon_{-\beta}(xy)
=\left(\varepsilon_{-\beta}(x^2y)\varepsilon_{\alpha}(xy)\right)\cdot\left(\varepsilon_{\alpha}(xy)\varepsilon_{-\beta}(xy)\right)\in B_A(4+2)=B_A(6).
\end{equation*}
This implies in particular that the ideal $(x^2-x)R$ is contained in $\varepsilon(A,-\beta,6)$ and hence according to \cite[Lemma~4.8(2)]{General_strong_bound}, the inclusion $(x^2-x)R\subset\varepsilon(A,\phi,18)$ holds for all $\phi\in C_2$ and $x\in R.$ But $R$ is the ring of integers in a global function field $K$ and consequently one can check that $k=[K:\mathbb{F}_q(T)]$ elements $x_1^2-x_1,\dots,x_k^2-x_k$ suffice to generate the ideal $I:=(x^2-x\mid x\in R).$ Thus $I\subset\varepsilon(A,\phi,18k)$ holds for all $\phi\in C_2.$ 

Next, note that $I$ has finite index in $R.$ As $R$ is a ring of integers of a global function field, all of its non-zero ideals have finite index. So it suffices to observe that there is an $x\in R$ such that $x^2-x\neq 0,$ say $x=T\in\mathbb{F}_q[T]\subset R.$

Pick a set $X\subset R$ of coset representatives of $I$ in $R.$ According to Theorem~\ref{main_thm}, the group ${\rm Sp}_4(R)$ is boundedly generated by root elements. So each element in ${\rm Sp}_4(R)$ can be written as a product of at most $L(k,q)\cdot 8=:V$ root elements. Let $C\in {\rm Sp}_4(R)$ be given and choose $a_1,\dots,a_V\in R$ as well as $\phi_1,\dots,\phi_V\in C_2$ such that 
\begin{equation*}
C=\prod_{i=1}^V\varepsilon_{\phi_i}(a_i).
\end{equation*}
Each element $a\in R$ can be written as $a=b+x$ for $b\in I$ and $x\in X.$ So choose $x_1,\dots,x_V\in X$ and $b_1,\dots,b_V\in I$ with  
$a_i=x_i+b_i$ for all $i=1,\dots,V.$ Then, we obtain
\begin{align}\label{char2_uniform_boundedness_tech_prop_eq2}
C=\prod_{i=1}^V\varepsilon_{\phi_i}(b_i)\varepsilon_{\phi_i}(x_i)=\varepsilon_{\phi_1}(b_1)\cdot[\prod_{i=2}^V\varepsilon_{\phi_i}(b_i)^{\varepsilon_{\phi_1}(x_1)\cdots\varepsilon_{\phi_{i-1}}(x_{i-1})}]
\cdot[\varepsilon_{\phi_1}(x_1)\cdots\varepsilon_{\phi_V}(x_V)]
\end{align}
But all the elements 
\begin{equation*}
\varepsilon_{\phi_1}(b_1),\{\varepsilon_{\phi_i}(b_i)^{\varepsilon_{\phi_1}(x_1)\cdots\varepsilon_{\phi_{i-1}}(x_{i-1})}\}_{2\leq i\leq V}
\end{equation*}
are elements of $N$ and there are only finitely many possibilities for the product  
\begin{equation*}
\varepsilon_{\phi_1}(x_1)\cdots\varepsilon_{\phi_V}(x_V).
\end{equation*}
Hence $N$ has finite index in ${\rm Sp}_4(R).$ On the other hand, if $C$ is in $N,$ then $\varepsilon_{\phi_1}(x_1)\cdots\varepsilon_{\phi_V}(x_V)$ is also an element of $N$. But there are only finitely many possibilities for $\varepsilon_{\phi_1}(x_1)\cdots\varepsilon_{\phi_V}(x_V)$, so there is an $M\in\mathbb{N}$ such that 
$\|\varepsilon_{\phi_1}(x_1)\cdots\varepsilon_{\phi_V}(x_V)\|_A\leq M$ holds for all the possible products $\varepsilon_{\phi_1}(x_1)\cdots\varepsilon_{\phi_V}(x_V)\in N.$ But we already know that all the elements 
\begin{equation*}
\varepsilon_{\phi_1}(b_1),\{\varepsilon_{\phi_i}(b_i)^{\varepsilon_{\phi_1}(x_1)\cdots\varepsilon_{\phi_{i-1}}(x_{i-1})}\}_{2\leq i\leq V}
\end{equation*}
are elements of $B_A(18k).$ Hence (\ref{char2_uniform_boundedness_tech_prop_eq2}) implies 
\begin{align*}
\|C\|_A=\|\varepsilon_{\phi_1}(b_1)\cdot[\prod_{i=2}^V\varepsilon_{\phi_i}(b_i)^{\varepsilon_{\phi_1}(x_1)\cdots\varepsilon_{\phi_{i-1}}(x_{i-1})}]
\cdot[\varepsilon_{\phi_1}(x_1)\cdots\varepsilon_{\phi_V}(x_V)]\|_A\leq 18kV+M
\end{align*}
and this finishes the proof.
\end{proof}


In order to prove Theorem~\ref{strong_bound_pos} for ${\rm Sp}_4$, we need the following technical lemma essentially due to Costa and Keller \cite{MR1162432}:

\begin{lemma}\cite[Theorem~4.7]{General_strong_bound},\cite[Theorem~2.6, 4.2, 5.1, 5.2]{MR1162432}\label{technical_keller_lemma}
Let $R$ be a commutative ring with $1$ and $A\in{\rm Sp}_4(R)$ be given. Then for each $x\in l(A),$ the product $\varepsilon_{2\alpha+\beta}(2x+x^2)\varepsilon_{\alpha+\beta}(x^2)$ is contained in any normal subgroup of ${\rm Sp}_4(R)$ containing $A.$
\end{lemma}

Finally, we can prove Theorem~\ref{strong_bound_pos} in the $\Phi=C_2$ and ${\rm char}(K)=2$-case:

\begin{proof}
Let $S$ be a finite, normally generating subset of ${\rm Sp}_4(R)$. To prove Theorem~\ref{strong_bound_original} for ${\rm Sp}_4(R)$, we will proceed in three steps. First, we will show that there is a natural number $M$ independent of $R$ such that 
\begin{equation*}
\|\varepsilon_{\alpha+\beta}(1)\varepsilon_{2\alpha+\beta}(1)\|_S\leq M|S|
\end{equation*}
holds. Second, we will use the second part of Proposition~\ref{char2_uniform_boundedness_tech_prop} to bound the normal subgroup $N$ generated by 
$\varepsilon_{\alpha+\beta}(1)\varepsilon_{2\alpha+\beta}(1)$ with respect to the norm $\|\cdot\|_S.$
Third, we will conclude from the first part of Proposition~\ref{char2_uniform_boundedness_tech_prop} and an argument similar to the proof of \cite[Proposition~3.7]{General_strong_bound} that ${\rm Sp}_4(R)$ is strongly bounded.

For the first step, let $X\in {\rm Sp}_4(R), y\in R$ and $i,j$ distinct elements of $\{1,\dots,4\}$ be given. Then observe that Lemma~\ref{technical_keller_lemma} and
${\rm char}(K)=2$, implies that 
\begin{equation*}
\varepsilon_{\alpha+\beta}(y^2x_{i,j}^2)\varepsilon_{2\alpha+\beta}(y^2x_{i,j}^2)\in\dl X\dr.
\end{equation*}
So using a first-order compactness argument similar to the one in the proof of \cite[Proposition~4.9]{General_strong_bound}, one then proves the existence of a natural number $W$ such that for all $y\in R$ and $X\in{\rm Sp}_4(R)$, one has 
\begin{equation*}
\|\varepsilon_{\alpha+\beta}(y^2x_{i,j}^2)\varepsilon_{2\alpha+\beta}(y^2x_{i,j}^2)\|_X\leq W
\end{equation*}
for all $i,j$ distinct elements of $\{1,\dots,4\}$. Possibly enlarging $W$, one can also show for all $y\in R$ and $X\in{\rm Sp}_4(R)$, that 
\begin{equation*}
\|\varepsilon_{\alpha+\beta}(y^2(x_{i,i}-x_{j,j})^2)\varepsilon_{2\alpha+\beta}(y^2(x_{i,i}-x_{j,j})^2)\|_X\leq W
\end{equation*}
for all $i,j$ distinct elements of $\{1,\dots,4\}.$

Next, observe that $S$ is a normally generating set and hence \cite[Corollary~3.11]{General_strong_bound} implies $\sum_{X\in S}l(X)=R.$ Thus there are elements 
\begin{equation*}
\{y_{ij}^{(X)},z_{ij}^{(X)}\}_{X\in S,1\leq i\neq j\leq 4}\subset R
\end{equation*}
with
\begin{equation*}
1=(\sum_{1\leq i\neq j\leq 4,X\in S}y_{ij}^{(X)}x_{ij})+(\sum_{1\leq i\neq j\leq 4,X\in S}z_{ij}^{(X)}(x_{i,i}-x_{j,j})).
\end{equation*}
But remember that ${\rm char}(K)=2$ and hence Freshman's dream implies
\begin{equation*}
1=(\sum_{1\leq i\neq j\leq 4,X\in S}(y_{ij}^{(X)})^2x_{ij}^2)+(\sum_{1\leq i\neq j\leq 4,X\in S}(z_{ij}^{(X)})^2(x_{i,i}-x_{j,j})^2).
\end{equation*}
Thus, enlarging $W$ again, we obtain 
\begin{align*}
B_S\left(W|S|)\right)\ni&[\prod_{1\leq i\neq j,X\in S}\varepsilon_{\alpha+\beta}((y_{ij}^{(X)})^2x_{ij}^2)\varepsilon_{2\alpha+\beta}((y_{ij}^{(X)})^2x_{ij}^2)]\\
&\ \ \ \cdot[\prod_{1\leq i\neq j,X\in S}\varepsilon_{\alpha+\beta}((z_{ij}^{(X)})^2(x_{ii}-x_{jj})^2)\varepsilon_{2\alpha+\beta}((z_{ij}^{(X)})^2(x_{ii}-x_{jj})^2)]\\
&=\varepsilon_{\alpha+\beta}((\sum_{1\leq i\neq j,X\in S}(y_{ij}^{(X)})^2x_{ij}^2)+(\sum_{1\leq i\neq j,X\in S}(z_{ij}^{(X)})^2(x_{ii}-x_{jj})^2))\\
&\ \ \ \cdot\varepsilon_{2\alpha+\beta}((\sum_{1\leq i\neq j,X\in S}(y_{ij}^{(X)})^2x_{ij}^2)+(\sum_{1\leq i\neq j,X\in S}(z_{ij}^{(X)})^2(x_{ii}-x_{jj})^2))\\
&=\varepsilon_{\alpha+\beta}(1)\varepsilon_{2\alpha+\beta}(1).
\end{align*}
This finishes the first step. So we obtain  
\begin{equation*}
\|\varepsilon_{\alpha+\beta}(1)\varepsilon_{2\alpha+\beta}(1)\|_S\leq W|S|.
\end{equation*}

For the second step, note that the normal subgroup $N$ generated by 
\begin{equation*}
A:=\varepsilon_{\alpha+\beta}(1)\varepsilon_{2\alpha+\beta}(1)
\end{equation*}
is bounded with respect to the norm $\|\cdot\|_A$ according to the second statement of Proposition~\ref{char2_uniform_boundedness_tech_prop}. Setting
$L(R):=\|N\|_A$, this implies
\begin{equation*}
\|N\|_S\leq\|N\|_A\cdot\|A\|_S\leq L(R)W|S|.
\end{equation*}

For the third step, observe that according to the first statement of Proposition~\ref{char2_uniform_boundedness_tech_prop}, the normal subgroup $N$ has finite index in ${\rm Sp}_4(R)$ and ${\rm Sp}_4(R)$ is boundedly generated by root elements. Hence, if we replace in the proof of \cite[Proposition~3.7]{General_strong_bound} the subgroup $N_{C_2}$ by $N$, we can find an $M(R)\in\mathbb{N}$ such that $\|{\rm Sp}_4(R)\|_S\leq M(R)+\|N\|_S$ and hence 
\begin{equation*}
\|{\rm Sp}_4(R)\|_S\leq M(R)+L(R)W|S|.
\end{equation*}
Hence ${\rm Sp}_4(R)$ is strongly bounded and this finishes the proof.
\end{proof}

\subsection{The case of $G_2$}

For this subsection, recall that the root system $G_2$ has $12$ roots and we can choose positive, simple roots $\alpha,\beta\in G_2$ such that $\alpha$ is short and $\beta$ is long and the set of positive roots given by the choice of $\alpha$ and $\beta$ as positive, simple roots is $G_2^+=\{\alpha,\beta,\alpha+\beta,2\alpha+\beta,3\alpha+\beta,3\alpha+2\beta\}.$ We want to also point out that the subset $\{\pm\beta,\pm (3\alpha+\beta), \pm (3\alpha+2\beta)\}$ of $G_2$ gives a root subsystem of $G_2$ consisting of all long roots of $G_2$ and isomorphic to $A_2.$ The proof of Theorem~\ref{strong_bound_pos} in case of $\Phi=G_2$ works in essentially the same manner as in the case of $\Phi=C_2$ before. The only main difference is the manner in which a finite index subgroup of $G_2(R)$ is obtained:

\begin{proposition}\label{char2_uniform_boundedness_tech_prop_g2}
Let $K$ be a global function field that is a finite extension of $\mathbb{F}_q(T)$ and field of constants $\mathbb{F}_q$ with $q$ a power of $2.$ Further, let $R:=\C O_K$ be the ring of integers of $K.$ Further, let $N$ be the normal subgroup of $G_2(R)$ generated by $A:=\varepsilon_{\beta}(1).$ Also let $\|\cdot\|_A:G_2(R)\to\mathbb{N}_0\cup\{+\infty\}$ be the conjugation generated word norm on $G_2(R)$ defined as in Defintion~\ref{conjugation_gen_word+basic_def}. Then 
\begin{enumerate}
\item{$N$ is a finite index subgroup of $G_2(R)$ and}
\item{the norm $\|\cdot\|_A$ has finite diameter on $N$.}
\end{enumerate}
\end{proposition}

\begin{proof}
First, note that according to \cite[Proposition~4.12(1a)]{General_strong_bound}, we obtain $\{\varepsilon_{\phi}(x)\mid x\in R,\phi\in G_2\text{ long}\}\subset B_A(2).$ Hence for $x,y\in R$ this implies 
\begin{align*}
B_A(4)\ni(\varepsilon_{\beta}(xy),\varepsilon_{\alpha}(1))=\varepsilon_{\alpha+\beta}(xy)\varepsilon_{2\alpha+\beta}(xy)\varepsilon_{3\alpha+\beta}(xy)
\varepsilon_{3\alpha+2\beta}(x^2y^2)
\end{align*}
Similarly, we obtain
\begin{align*}
B_A(4)\ni(\varepsilon_{\beta}(y),\varepsilon_{\alpha}(x))=\varepsilon_{\alpha+\beta}(xy)\varepsilon_{2\alpha+\beta}(x^2y)\varepsilon_{3\alpha+\beta}(x^3y)
\varepsilon_{3\alpha+2\beta}(x^3y^2)
\end{align*}
Thus we obtain as ${\rm char}(K)=2$ that
\begin{align*}
B_A(8)\ni(\varepsilon_{\beta}(xy),\varepsilon_{\alpha}(1))\cdot(\varepsilon_{\beta}(y),\varepsilon_{\alpha}(x))
=\varepsilon_{2\alpha+\beta}((x+x^2)y)\varepsilon_{3\alpha+\beta}(xy+x^3y)\varepsilon_{3\alpha+2\beta}(x^2y^2+x^3y^2)
\end{align*}
But we already know that $\varepsilon_{3\alpha+\beta}(xy+x^3y)\varepsilon_{3\alpha+2\beta}(x^2y^2+x^3y^2)\in B_A(4)$ as both $3\alpha+\beta$ and $3\alpha+2\beta$ are long roots. Hence $\varepsilon_{2\alpha+\beta}((x+x^2)y)\in B_A(12)$ holds for all $x,y\in R.$ But now ${\rm vn}_2(R)=(x^2-x\mid x\in R)\subset\varepsilon(A,\phi,24k)$ for all $\phi\in G_2$ follows from $R$ being the ring of integers of $K$, the same way as in the proof of Proposition~\ref{char2_uniform_boundedness_tech_prop}. The rest of the proof works the same way as the rest of the proof of Proposition~\ref{char2_uniform_boundedness_tech_prop}, so we are going to skip it.  
\end{proof}

Now, we can prove Theorem~\ref{strong_bound_pos} however the proof is virtually identical to the proof in the ${\rm Sp}_4$-case, so we are skipping it, too beyond a short remark: Let $S$ be a finite, normally generating set of $G_2(R).$ Then the only major difference is the use of \cite[Proposition~3.9]{General_strong_bound} instead of Lemma~\ref{technical_keller_lemma} to deduce from $l(S)=R$ that there is a $W$ such that $\varepsilon_{\beta}(1)$ is contained in $B_S(W|S|).$

We also want to note the following corollary of the proof of Theorem~\ref{strong_bound_pos}:

\begin{corollary}\label{characterization_generation}
Let $K$ be a global function field that is a finite extension of $\mathbb{F}_q(T)$ with ring of integers $R:=\C O_K$ and field of constants $\mathbb{F}_q.$ Further, let $T_R$ be the set of all maximal ideals $\mathfrak{P}$ of $R$ with $R/\mathfrak{P}=\mathbb{F}_2$ and let $\Phi=C_2$ or $G_2.$ Then a subset $S$ of $G(\Phi,R)$ normally generates $G(\Phi,R)$ precisely iff the following two conditions hold: First, $l(S)=R$ holds and second $S$ maps to a normally generating set of $\prod_{\mathfrak{P}\in T_R}G(\Phi,R/\mathfrak{P})$ under the map 
\begin{equation*}
G(\Phi,R)\to\prod_{\mathfrak{P}\in T_R}G(\Phi,R/\mathfrak{P}),X\mapsto\prod_{\mathfrak{P}\in T_R}\pi_{\mathfrak{P}}(X).
\end{equation*} 
\end{corollary}

\begin{proof}
First, assume that $S$ normally generates $G(\Phi,R).$ Then the condition $l(S)=R$ follows from \cite[Corollary~3.11]{General_strong_bound}. Second, note that the set $T_R$ is finite: To this end, set $I:={\rm vn}_2(R):=(x^2-x\mid x\in R)$ and let $S_R$ be the set of prime divisors of $I$. For each $\mathfrak{P}$, choose a natural number $a_{\mathfrak{P}}$ such that ${\rm vn}_2(R)=\prod_{\mathfrak{P}\in S_R}\mathfrak{P}^{a_{\mathfrak{P}}}.$ Then the set $T_R$ is finite, if $S_R=T_R$ holds. But the following chain of equivalences proves this claim. Let $J$ be a maximal ideal of $R$. Then:
\begin{align*}
J\in T_R&\iff R/J=\mathbb{F}_2\iff(\forall x\in R:x\in J\text{ or }x-1\in J)\iff(\forall x\in R:x(x-1)\in J)\\
&\iff(I\subset J)\iff J\in S_R.  
\end{align*} 
Hence by the Chinese Remainder Theorem, the group $\prod_{\mathfrak{P}\in T_R}G(\Phi,R/\mathfrak{P})$ is a quotient of $G(\Phi,R)$ and so $S$ maps to a normally generating set of $\prod_{\mathfrak{P}\in T_R}G(\Phi,R/\mathfrak{P})$. 

Second, assume that the two conditions mentioned in the corollary hold. For the sake of brevity, we are going to restrict ourselves to the case of $\Phi=C_2.$ Then as seen in the proof of Theorem~\ref{strong_bound_original} for this case, one can deduce from $l(S)=R$ that $\{\varepsilon_{\phi}(y)\mid y\in I,\phi\in C_2\}$ is contained in the normal subgroup of ${\rm Sp}_4(R)$ normally generated by $S$. But using Bass, Milnor and Serre solution of the congruence subgroup property \cite[Theorem~3.6]{MR244257}, the subgroup normally generated by the set $\{\varepsilon_{\phi}(y)\mid y\in I,\phi\in C_2\}$ is the kernel of the quotient map  
\begin{equation*}
\pi_I:{\rm Sp}_4(R)\to {\rm Sp}_4(R/I),X\mapsto\pi_{I}(X).
\end{equation*} 
Hence it suffices to show that $S$ maps onto a normally generating set of ${\rm Sp}_4(R/I).$ But note that all the $a_{\mathfrak{P}}$ mentioned in the first part of the proof are $1.$ This follows because $I={\rm vn}_2(R)$ is a radical ideal: Assume that $x\in\sqrt{I}$ is given, that is $x^n\in I$ for some $n\in\mathbb{N}.$ But for $n\geq 2,$ the difference $x^{n-1}-x^n$ is an element of $I$ and hence so is $x^{n-1}.$ Thus $x\in I$ follows by induction. But non-zero, radical ideals of a Dedekind domain are divided by their prime divisors precisely once. This finishes the proof.
\end{proof}

\section{Lower bounds on $\Delta_l$}\label{lower_bound_section}

In this section, we prove Corollary~\ref{lower_bounds} and Corollary~\ref{lower_bounds_rank2}. First, we recall the following theorem essentially due to Kedra, Libman and Martin:

\begin{theorem}\cite[Proposition~6.1]{General_strong_bound},\cite[Theorem~6.1]{KLM}\label{lower_bounds_number_fields}
Let $\Phi$ be an irreducible root system of rank at least $2$ and let $R$ be a Dedekind domain with finite class number and infinitely many maximal ideals such that 
$G(\Phi,R)$ is boundedly generated by root elements. Further assume that $2$ is a unit in $R$ if $\Phi=C_2$ or $G_2.$ Then $\Delta_l(G(\Phi,R))\geq l$ for all $l\in\mathbb{N}.$ 
\end{theorem}

However, we know that any ring of integers in a global function field $K$ is a Dedekind domain with finite class number \cite[Lemma~5.6]{MR1876657} and infinitely many maximal ideals and so Corollary~\ref{lower_bounds} follows immediately. 

The proof of Corollary~\ref{lower_bounds_rank2} works in a similar manner to the proof of \cite[Theorem~6.3]{General_strong_bound} but invoking Corollary~\ref{characterization_generation} instead of \cite[Corollary~3.11]{General_strong_bound}. To this end let 
\begin{equation*}
T_R=\{\mathfrak{P}_1,\dots,\mathfrak{P}_{r(R)}\}
\end{equation*}
be the set of prime ideals in $R$ with field of residue $\mathbb{F}_2.$ The second part of Corollary~\ref{lower_bounds_rank2} is proven as follows:
First, note that we have an epimorphism $G(\Phi,R)\to\prod_{\mathfrak{P}\in T_R}G(\Phi,R/\mathfrak{P})=G(\Phi,\mathbb{F}_2)^{r(R)}.$ But \cite[Lemma~6.5]{General_strong_bound} implies that $G(\Phi,\mathbb{F}_2)$ has the quotient $\mathbb{F}_2.$ Hence $G(\Phi,R)$ has the quotient $\mathbb{F}_2^{r(R)},$ which makes it impossible to normally generate $G(\Phi,R)$ with less than $r(R)$ elements. This finishes the proof of the second part of Corollary~\ref{lower_bounds_rank2}.

For the proof of the first part of Corollary~\ref{lower_bounds_rank2}, choose for $l\geq r(R)$ primes $\mathfrak{Q}_{r(R)+1},\dots,\mathfrak{Q}_l$ distinct from the elements of $T_R.$ Next, choose $C$ as the classnumber of $R$ and pick elements $y_1,\dots,y_k$ of $R$ with
\begin{align*}
&y_i\text{ with }y_iR=(\mathfrak{P}_1\cdots\hat{\mathfrak{P}_i}\cdots\mathfrak{P}_{r(R)}\cdot\mathfrak{Q}_{r(R)+1}\cdots\mathfrak{Q}_l)^C\text{ for }1\leq i\leq r(R)\text{ and}\\
&y_i\text{ with }y_iR=(\mathfrak{P}_1\cdots\mathfrak{P}_{r(R)}\cdot\mathfrak{Q}_{r(R)+1}\cdots\hat{\mathfrak{Q}_i}\cdots\mathfrak{Q}_l)^C\text{ for }r(R)+1\leq i\leq l.
\end{align*}
Here the hat denotes the omission of the corresponding prime ideal. Then choose $\phi$ as a short root in $\Phi$ and consider $S:=\{\varepsilon_{\phi}(y_i)\mid 1\leq i\leq l\}.$ We leave it as an exercise to the reader to check using Corollary~\ref{lower_bounds_rank2} that $S$ normally generates $G(\Phi,R)$ and that $\|G(\Phi,R)\|_S\geq l$ holds as the proofs are now virtually identical to the proof of \cite[Theorem~6.3]{General_strong_bound}. These two claims together imply $\Delta_l(G(\Phi,R))\geq\|G(\Phi,R)\|_S\geq l$ and so finish the proof. 

	
\section{Closing remarks}

A couple of remarks are in order. First, concerning Theorem~\ref{main_thm}: Its proof can be easily generalized to also encompass the situation where the ring $\C O_K$ is replaced by one of its localizations. Next, the bounds on bounded elementary generation appearing in Theorem~\ref{main_thm} depend on $q$ and $k=[K:\mathbb{F}_q(T)].$ The dependence on $q$ is almost certainly an artifact of the proof strategy chosen though; for example the partial result by Nica, Theorem~\ref{Nica_thm}, mentioned in the introduction provides a bounded generation result for ${\rm SL}_n(\mathbb{F}_q[T])$ and $n\geq 3$ that does not depend on $q.$ As mentioned though Nicas result are essentially an adaption of arguments by Carter and Keller \cite{MR704220} which are based on power residue symbols. Further, the Carter-Keller-argument depends on the ring being a ring of algebraic integers and especially containing the integers $\mathbb{Z}$ at various places and an adaptation of said argument seems likely to be more involved. In contrast, Morris' argument from \cite{MR2357719} essentially has to be modified at just one place, namely for Proposition~\ref{positive_char_exp_prep}, and then will go through without much further modifications. All that being said, adapting the full argument by Carter and Keller would likely eliminate the dependence on $q$.

But a faithful adaption of the Carter-Keller arguments would probably not manage to remove the dependence on the field $K$ in question. However, this dependence is likely also an artifact: Namely, in the number field case, this dependence on $K$ was eliminated in all cases but the imaginary quadratic field-case by the previously mentioned result concerning bounded elementary generation for ${\rm SL}_2$ due to Rapinchuk, Morgan and Sury \cite{MR3892969} and then using Proposition~\ref{Tavgen_reduction} to generalize it to the higher rank cases. However, these arguments used crucially that the ring of algebraic integers in question had infinitely many units and this is simply not the case for any ring of integers $\C O_K$ in a global function field \cite[Chapter~14, Corollary~1]{MR1876657}. However, I still believe that bounded elementary generation should not depend on $K$ in the function field case: Namely, the study of the congruence kernel by Bass, Milnor and Serre gave the same description for the congruence kernel for both the function field case and the non-imaginary-quadratic number field case.


Second, concerning strong boundedness and Theorem~\ref{strong_bound_pos}: In case of $\Phi\neq C_2$ or $G_2,$ the shape of bounds on $\Delta_l$ is essentially identical to the corresponding results for number fields, Theorem~\ref{alg_integers_strong_bound}, Theorem~\ref{lower_bounds_number_fields} and \cite[Theorem~6.3]{General_strong_bound}, and assuming the dependency of bounded generation by root elements on the field $K$ in question can be eliminated then it will have precisely the same form. More interesting is the case of $\Phi=C_2$ or $G_2$ and $K$ of characteristic $2:$ As we have seen in the proofs of Proposition~\ref{char2_uniform_boundedness_tech_prop} and \ref{char2_uniform_boundedness_tech_prop_g2}, we introduced a dependence on $K$ in the form of the number of generators of the booleanizing ideal ${\rm vn}_2(R)$ of the ring of integers in $K;$ that is a dependence on $k:=[K:\mathbb{F}_q(T)].$ However, while the dependency on $K$ can not be removed entirely, one can by assuming an improved version of Theorem~\ref{main_thm}, modify the proof of Proposition~\ref{char2_uniform_boundedness_tech_prop} and \ref{char2_uniform_boundedness_tech_prop_g2} slightly to make the dependence on $K$ asymptotically unimportant in the sense that $\limsup_{l\to\infty}\Delta_l(G(\Phi,R))/l$ would have an upper bound independent from $K.$ Again, we expect a similar behavior in the number field case.
 
Last, let us very briefly talk about a different application of Theorem~\ref{main_thm} we only learned about recently: A Theorem by Segal and Kent \cite[Theorem~1.1]{segal2021defining} states that for any commutative ring $R$ the group $G(\Phi,R)$ is bi-interpretable with the ring $R$ if the group $E(\Phi,R)$ is boundedly generated by root elements and some additional technical assumptions hold. Here \textit{bi-interpretability} means roughly that group $G(\Phi,R)$ can be described in the first order language of the ring $R$ and the ring $R$ can be described in the first order language of the group $G=G(\Phi,R)$ and said descriptions are in some sense compatible with one another. Thus our Theorem~\ref{main_thm} yields another possible application of Segal and Kent's Theorem.


\bibliography{bibliography}
\bibliographystyle{plain}

\end{document}